\documentclass[reqno]{amsart}

\usepackage{amscd, color, mathrsfs,enumerate}

\usepackage{amsfonts,latexsym,amstext}
\usepackage{amsmath, amssymb,xcolor}
\usepackage{amsthm}
\usepackage{bm}

\newtheorem{theorem}{Theorem}[section]
\newtheorem{lemma}[theorem]{Lemma}
\newtheorem{proposition}[theorem]{Proposition}
\newtheorem{definition}[theorem]{Definition}
\newtheorem{example}[theorem]{Example}
\newtheorem{hypotheses}[theorem]{Hypotheses}

\newtheorem{remark}[theorem]{Remark}

\numberwithin{equation}{section}

\def\sqr#1#2{{\vcenter{\vbox{\hrule height .#2pt \hbox{\vrule
 width .#2pt height#1pt \kern#1pt \vrule
width .#2pt} \hrule height .#2pt}}}}

\def\ds{\begin{displaystyle}}
\def\eds{\end{displaystyle}}

\def\<{\langle }
\def\>{\rangle }

\def\R{\mathbb R}
\def\N{\mathbb N}

\allowdisplaybreaks

\begin{document}

\title[Young equations with singularities]{Young equations with singularities}

\author[D. Addona]{Davide Addona}
\address{D.A.: Dipartimento di Scienze Matematiche, Fisiche e Informatiche, Plesso di Matema\-ti\-ca, Universit\`a degli Studi di Parma, Viale Parco Area delle Scienze 53/A, I-43124 Parma, Italy}
\author[L. Lorenzi]{Luca Lorenzi}
\address{L.L.: Dipartimento di Scienze Matematiche, Fisiche e Informatiche, Plesso di Matema\-ti\-ca, Universit\`a degli Studi di Parma, Viale Parco Area delle Scienze 53/A, I-43124 Parma, Italy}
\author[G. Tessitore]{Gianmario Tessitore}
\address{G.T.: Dipartimento di Matematica e Applicazioni, Universit\`a degli Studi di Milano-Bicocca, Milano, Via R. Cozzi 55, I-20126 Milano Italy}
\email{davide.addona@unipr.it}
\email{luca.lorenzi@unipr.it}
\email{gianmario.tessitore@unimib.it}

\maketitle
\begin{abstract}
In this paper we prove existence and uniqueness of a mild solution to the Young equation $dy(t)=Ay(t)dt+\sigma(y(t))dx(t)$, $t\in[0,T]$, $y(0)=\psi$. Here, $A$ is an unbounded operator which generates a semigroup of bounded linear operators  $(S(t))_{t\geq0}$ on a Banach space $X$, $x$ is a real-valued $\eta$-H\"older continuous. Our aim is to reduce, in comparison to \cite{GLT06}  and \cite{ALT} (see also \cite{DGT12,GT10}), the regularity requirement on the initial datum $\psi$ eventually dropping it.

The main tool is the definition of a sewing map for a new class of increments which allows the construction of a Young convolution integral in a general interval $[a,b]\subset \R$ when the $X_\alpha$-norm of the function under the integral sign blows up approaching $a$ and $X_{\alpha}$ is an intermediate space between $X$ and $D(A)$. 
\end{abstract}

$$ $$



\section{Introduction}
This paper is devoted to the study of the existence, uniqueness and regularity of the mild solution to the non-linear evolution problem
\begin{align}
\label{intro_nonlin_pb}
dy(t)=Ay(t)dt+\sigma(y(t))dx(t), \qquad t\in[0,T], \quad y_0=\psi,   
\end{align}
when the operator $A:D(A)\subset X\rightarrow X$ generates a semigroup of linear bounded operators $(S(t))_{t\geq0}$ on $X$, with smoothing effects, and $x$ is a $\eta$-H\"older continuous function with values in a suitable space. By mild solution of \eqref{intro_nonlin_pb} we mean a function $y$ which satisfies the equation
\begin{align}
\label{intro_mild_sol}
y(t)=S(t)\psi+\int_0^tS(t-r)\sigma(y(r))dx(r), \qquad t\in[0,T].
\end{align}
We stress that, since $x$ has not finite variation, the integral appearing in the above formula has to be intended in a non-obvious sense. 
Integrals such as
\begin{align}
\label{intro_stie_int}
    \int_s^t y(r)dx(r),
\end{align}
when neither $x$ or $y$ have bounded variation, have been introduced by L.C. Young in \cite{You36} when $x$ and $y$ have finite $p$-variation\footnote{A function $z:[a,b]\to X$ has finite $r$-variation if the supremum of $\sum_{j=1}^n\|z(t_j)-z(t_{j-1})\|^r$ over all the partitions $\Pi=\{a=t_0<t_1<\ldots<t_n=b\}$ of $[a,b]$ is finite.} and $q$-variation, respectively, with $\frac{1}{p}+\frac{1}{q}>1$ (equivalently, when $y$ is $\beta$-H\"older continuous and $x$ is $\eta$-H\"older continuous with $\beta+\eta>1$), and are therefore called Young integrals.  In \cite{Lyo98}, T. Lyons extended the definition of the integral in \eqref{intro_stie_int} when $\frac{1}p+\frac1q<1$ (i.e., when $\beta+\eta<1$) by adding some ``structure'' to the irregular path $x$ giving birth to the rough-paths theory. A different approach has been developed by M. Gubinelli in \cite{G04}, where the rough integral is defined as the unique solution of an algebraic problem under some analytic conditions.
Here, we follow this approach and we limit ourselves to solve \eqref{intro_mild_sol} in the Young case.

\medskip
The evolution equation \eqref{intro_nonlin_pb} in infinite dimensional spaces was first addressed in \cite{GLT06}. It is clear that the key point consists of the definition of the convolution integral
\begin{align}
\label{intro_conv_int}
\int_s^tS(t-r)\varphi(r)dx(r)    
\end{align}
for suitable functions $\varphi$ and H\"older continuous path $x$ defined on $[0,T]$. This is done in  \cite{GLT06,GT10} by adapting the sewing map approach to the case of convolutions. Given $A:[a,b]^2_<\rightarrow X$, where $[a,b]^2_<:=\{(s,t):a\leq s\leq t\leq b\}$, the 
authors of \cite{GLT06,GT10} defined the `convolutional' increment $(\hat\delta A)(r,s,t):=A(r,t)-A(s,t)-S(t-s)A(r,s)$ for $a\leq r\leq s\leq t\leq b$ and proved  that, if $\|(\hat\delta A)(r,s,t)\|_X\leq C|t-r|^\mu$ for every $a\le r\le s\le t\le b$ and some constants $C>0$ and $\mu>1$, then the limit
\begin{align}\label{intro_Rieman_1}
R(s,t):=\lim_{|\Pi(s,t)|\rightarrow 0}\sum_{i=1}^n S(t-t_i)A(t_{i-1},t_i),   \qquad\;\, 
(s,t)\in[a,b^2]_<,
\end{align}
exists (here $\Pi(s,t)$ is a partition of $[s,t]$ and $|\Pi(s,t)|$ is its mesh).
The convolution integral can then be defined by:
\begin{align}
\label{intro_def_conv_int_gub}
\mathscr{I}(s,t):=A(s,t)-R(s,t), \qquad\;\, (s,t)\in[a,b]^2_<,  \end{align}
and, with the special choice $A(s,t)=S(t-s)\varphi(s)(x(t)-x(s))$ for every $a\le r\le s\le t\le b$, where $x$ is a real-valued $\eta$-H\"older continuous function and $\varphi:[a,b]\rightarrow X$ verifies $\|\varphi(t)-S(t-s)\varphi(s)\|_X\leq C|t-s|^\beta$ for every $a\le s\le t\le b$ and some $\beta>0$, with $\eta+\beta>1$, the above construction  suggests to set
\begin{align*}
\int_s^t S(t-r)\varphi(r)dx(r):=\mathscr{I}(s,t),\qquad\;\, (s,t)\in[a,b]^2_<.
\end{align*}
Once the convolution integral in \eqref{intro_conv_int} is defined, its properties allow proving existence and uniqueness of the mild solution to \eqref{intro_nonlin_pb} by means of a fixed-point argument. We underline that a crucial assumption in  \cite{GLT06,GT10} is that the initial datum $\psi$ belongs to an intermediate space $X_\alpha$ between $X$ and $D(A)$ with $\eta+\alpha>1$. Moreover, the solution lives in the same space $X_\alpha$ with no gain of regularity. On the contrary, in the classical case, see for instance \cite{lor-rha} or  \cite{LU95}, i.e., when $x$ is smooth and $A$ generates an analytic semigroup, then no regularity assumptions on $\psi$ are needed (in particular, we can choose $\psi\in X$). Moreover, if  the intermediate spaces are interpolation spaces of indexes $\alpha$ and $\infty$ and the initial datum belongs to any of such spaces, then $y(t)\in D(A)$ for every $t\in(0,T]$ and $y$ satisfies \eqref{intro_nonlin_pb} in the original differential form.

This consideration was the starting point of \cite{ALT}, where the second of the limitations described above is tackled. To be more specific, in \cite{ALT} equation \eqref{intro_mild_sol} is studied when $x$ is a real-valued $\eta$-H\"older continuous path with $\eta>\frac12$. We showed that the mild solution $y$ to \eqref{intro_nonlin_pb} becomes more regular as soon as it leaves 0, that is, still under the assumption that $\psi\in X_{\alpha}$ with $\alpha+\eta>1$, we proved that $y(t)\in D(A)$ for every $t\in(0,T]$. Moreover, an estimate of the blow-up of $\|y(t)\|_{D(A)}$ when $t$ approaches $0$ has been proved. Thanks to this result, we also provided an integral representation of the solution $y$, which yields a chain rule for $F\circ y$ when $F$ is a smooth function.

\medskip
The aim of the present paper is to overcome the first of the two limitations described above, namely the regularity request  $\psi\in X_\alpha$, with $\alpha+\eta>1$, on the initial datum $\psi$. We investigate here the properties of the mild equation  \eqref{intro_nonlin_pb} when $\psi$ belongs to a larger space $X_\theta$ with $0\leq\theta<\alpha$. Under such weaker assumptions, the function $t\mapsto \|S(t)\psi\|_\alpha$ has a singularity at $0$, thus the definition of the convolution integral \eqref{intro_conv_int} has to be extended to the case when the function $f$ has a singularity at $0$. First,  we modify the construction of the sewing map introducing a different notion of increment:  given $g:[a,b]^2_<\to X$, we set $(\delta_S g)(r,s,t):=S(s-r)g(r,t)+g(s,t)-S(s-r)g(r,s)$ for every $a\leq r\leq s\leq t\leq b$ and prove, see Proposition \ref{prop:new_sew_map},  that the limit
\begin{align}\label{intro_Rieman_2}
R(s,t):=\lim_{|\Pi(s,t)|\rightarrow 0}\sum_{i=1}^n S(t-t_{i-1})g(t_{i-1},t_i), \qquad\;\, (s,t)\in[a,b]^2_<,  \end{align}
exists whenever  $\|(\delta_S g)(r,s,t)\|_X\leq C|t-r|^\mu$ for every $a\le r\le s\le t\le b$ and some constants $C>0$ and $\mu>1$.
We notice that in \eqref{intro_Rieman_2}, differently from what happens in \eqref{intro_Rieman_1}, the semigroup $(S(t))_{t\ge 0}$ is never evaluated at $0$, since $t_{i-1}<t=t_n$ for every $i=1,\ldots,n$. This allows exploiting the regularizing properties of $(S(t))_{t\ge 0}$ in order to deduce regularizing properties for the convolution integral 
$\mathscr{I}$ defined as in \eqref{intro_def_conv_int_gub}. 

Then, see Theorem \ref{thm:sew_map_sing_int}, we go further and, taking advantage of the regularity of the convolution integral, we extend the definition of \eqref{intro_conv_int} when the
continuous function $\varphi:(a,b]\to X_{\alpha}$ has a singularity of order $\gamma>0$ at $a$ and there exist constants $C>0$, $\rho>1-\eta$, $\beta\in[0,\alpha]$ and $\gamma\in(0,1)$ such that  $\|\varphi(t)-S(t-s)\varphi(s)\|_{X_{\beta}}\leq C|t-s|^{\rho}|s-a|^{-\gamma}$ for every $a<s\le t\le b$. Roughly speaking, we use Proposition \ref{prop:new_sew_map} to define the value of $\int_s^t S(t-r)\varphi(r)dx(r)$ when
$(s,t)\in [a+\theta,b]^2_<$, for every $\theta\in(0,b-a)$, and then we exploit its regularity in order to extend it up to $s=a$. We also prove estimate \eqref{stima-utilee} which is an essential tool when dealing with equation \eqref{intro_mild_sol}.

Eventually, taking into account estimates  \eqref{stima-utilee} on the singular convolution integral, we are in a position prove existence, uniqueness and smoothness of the mild solution to \eqref{intro_nonlin_pb} with $x\in C^\eta([0,T])$, $\eta\in(1/2,1)$, and general initial datum $\psi\in X_\theta$ with, possibly, $\eta+\theta<1$.  This is firstly done in Subsection \ref{subsect-3.1}, when, besides other technical assumptions (see Hypotheses \ref{ip_nonlineare}), $\sigma$ is globally Lipschitz continuous on a regular space $X_{\alpha}$ with $\alpha + \eta >1$. Then,  see Subsection  \ref{subsect-3.2}, we allow $\sigma$ to be only locally Lipschitz continuous in $X_{\alpha}$. This framework turns out to be suitable to treat parabolic equations in spaces of continuous functions, see Example \ref{example-3.14}. Finally, in Subsection \ref{subsect-3.3}, we  drop the above-mentioned local Lipschitzianity request on $\sigma$ and only require that $\sigma' $ is locally Lipschitz continuous from $X_{\alpha}$ to the larger space $X$. This allows treating, for instance, one-dimensional parabolic equations in $L^2$-spaces, see Example \ref{example-3.18}.

  \smallskip
The paper is organized as follows.
In Section \ref{sect-2}, we first state the general assumptions and describe the functional spaces that will be used in the paper. In Subsection \ref{subsect-2.2}, we present our version of the sewing map,
construct the convolution integral and study its regularity. Next, in Subsection \ref{subsect-2.3} we extend the above results to the case of functions with singularities. Finally, 
Section \ref{sect-3} is devoted to the study of equation \eqref{intro_mild_sol}.

\medskip
\paragraph{\bf Notation.} 
For every $a,b\in\R$, with $a<b$, we set  $[a,b]^n_<:=\{(t_1,\ldots,t_n):a\leq t_1\le t_2\le\ldots\le t_n\leq b\}$ and $(a,b]^n_<:=\{(t_1,\ldots,t_n):a<t_1\le t_2\le\ldots\le t_n\leq b\}$. By $\Pi(a,b)$ we denote a partition of the interval $[a,b]$, i.e., $\Pi(a,b):=\{t_0=a<t_1<\ldots<t_n=b\}$ for some $n\in\N$. $|\Pi(a,b)|$ is the amplitude of the partition, i.e., $|\Pi(a,b)|:=\max_{i=1,\ldots,n}|t_i-t_{i-1}|$, and the limit $|\Pi(a,b)|\rightarrow 0$ is meant as limit on direct sets.
Finally, we denote by $B:(0,\infty)\times (0,\infty)\to\R$ the Euler $\beta$-function defined by $B(\alpha,\beta)=\displaystyle\int_0^1t^{\alpha-1}(1-t)^{\beta-1}dt$ for every  $\alpha,\beta\in (0,\infty)$.

\section{The sewing map and the convolution integral for singular functions}
\label{sect-2}
\subsection{Main assumptions, spaces of functions and increments}
\label{subsect-2.1}
\begin{hypotheses}
\label{hyp-main}
\begin{enumerate}[\rm (i)]
\item
$A:D(A)\subset X\to X$ is the generator a semigroup of bounded operators $(S(t))_{t\geq0}$ on the Banach space $X$.
\item
For every $\lambda\in [0,3)$, there exists a space $X_{\lambda}$ $($with the convention that $X_0=X$ and $X_1=D(A))$ such that if $\beta<\lambda$ then $X_{\lambda}$ is continuously embedded into $X_\beta$. We denote by $K_{\lambda,\beta}$ a positive constant such that $\|x\|_{\beta}\le K_{\lambda,\beta}\|x\|_{\lambda}$ for every $x\in X_{\lambda}$;
\item
for every $\zeta,\lambda\in [0,3)$, with $\zeta\leq\lambda$, and $\mu,\nu\in[0,1]$, with $\mu>\nu$, there exist positive constants $L_{\zeta,\lambda,T}$, and $C_{\mu,\nu,T}$, which depend on $T$, such that\footnote{When no confusion may arise, we do not indicate the dependence of the constants on $T$.}
\begin{align}
\left\{
\begin{array}{ll}
(a)\ \|S(t)\|_{\mathscr{L}(X_\zeta, X_{\lambda})}\leq L_{\zeta,\lambda,T} t^{-\lambda+\zeta},\\[1mm]
(b) \ \|S(t)-I\|_{\mathscr{L}(X_\mu,X_\nu)}\leq C_{\mu,\nu,T} t^{\mu-\nu},
\end{array}
\right.
\label{stime_smgr}
\end{align}
for every $t\in (0,T]$;
\item 
$S(t)$ is injective for every $t\geq 0$.
\end{enumerate}
\end{hypotheses}

\begin{remark}
\label{rmk-2.2}
{\rm 
Conditions \eqref{stime_smgr} imply that the function $t\mapsto S(t)$ is continuous in $(0,+\infty)$ with values in $\mathscr{L}(X)$. Indeed, fix $t_0>0$, $t\in (t_0/2,2t_0)$ and $\lambda>0$. Using \eqref{stime_smgr}, we can estimate
\begin{align*}
\|S(t)x-S(t_0)x\|_X=&\|(S(t\vee t_0-t\wedge t_0)-I)S(t_0\wedge t)x\|_X\\
\leq & C_{\mu,0,2t_0}|t-t_0|^{\lambda}\|S(t_0\wedge t)x\|_{X_{\lambda}} \\ 
\leq & 2^{\lambda}C_{\lambda,0,2t_0}L_{0,\lambda,2t_0}t_0^{-\lambda}\|x\|_X|t-t_0|^{\lambda}
\end{align*}
for every $x\in X$ so that the function $t\mapsto S(t)$ is continuous at $t_0$ with values in $\mathscr{L}(X)$.
On the other hand, we stress that we do not require continuity of the function $t\mapsto S(t)x$ at $t=0$.}
\end{remark}

\begin{example}
If $A$ is a sectorial operator on $X$, then Hypotheses $\ref{hyp-main}$ are satisfied with $X_{\lambda}=D_A(\lambda,\infty)$ for every $\lambda\in (0,2)$, $X_2=D(A^2)$ and $X_{\lambda}=\{x\in D(A^2): A^2x\in D_A(\lambda-2,\infty)\}$ if $\lambda\in (2,3)$.
\end{example}

% \begin{definition}
% We say that a function $f:[a,b]_{<}^n\to X_{\lambda}$, for some $\lambda\in [0,2)$, is continuous on $[a,b]^n_<$ if it is continuous on $\{(t_1,\ldots,t_n):\ a\leq t_1\leq t_2\leq\ldots \leq t_n\leq b\}$ and $f(t_1,\ldots,t_n)=0$ if $t_i=t_{i+1}$ for some $i\in\{1,\ldots,n-1\}$.
% \end{definition}

Let $\mathfrak a(s,t):=S(t-s)-{\rm Id}_X$ for every $(s,t)\in[a,b]^2_<$. Following \cite{GT10}, for every $n\in\N$ we introduce the operator $\hat\delta_n:C([a,b]^n_<;X_\lambda)\rightarrow C([a,b]^{n+1}_<;X_\lambda)$ defined by
\begin{align*}
(\hat\delta_nf)(t_1,\ldots,t_{n+1})
:=& (\delta_nf)(t_1,\ldots,t_{n+1})-\mathfrak a(t_n,t_{n+1})f(t_1,\ldots,t_n) \\
:=& \sum_{j=1}^{n+1}(-1)^{n-j}f(t_1,\ldots,t_{n+1})^{\wedge j}-\mathfrak a(t_n,t_{n+1})f(t_1,\ldots,t_n) \\
:= & \sum_{j=1}^{n}(-1)^{n-j}f(t_1,\ldots,t_{n+1})^{\wedge j}-S(t_{n+1}-t_n)f(t_1,\ldots,t_n)
\end{align*}
for every $f\in C([a,b]^n_<;X_\lambda)$ and $(t_1,\ldots,t_{n+1})\in[a,b]^{n+1}_<$, where, for every $j=1,\ldots,n+1$, $(t_1,\ldots,t_{n+1})^{\wedge j}$ is the vector, with $n$ components, obtained from $(t_1,\ldots,t_{n+1})$ erasing the $j$-th element. In particular, if $n=1$, then we have 
\begin{eqnarray*}
(\hat\delta_1f)(s,t)=f(t)-f(s)-\mathfrak{a}(s,t)f(s)=f(t)-S(t-s)f(s),\qquad\;\,(s,t)\in[a,b]^2_<,
\end{eqnarray*}
and, if $n=2$, then we have
\begin{align*}
(\hat\delta_2 f)(r,s,t):=&
f(r,t)-f(r,s)-f(s,t)-\mathfrak a(s,t)f(r,s)\\
=&f(r,t)-f(s,t)-S(t-s)f(r,s)
\end{align*}
for every $(r,s,t)\in[a,b]^3_<$.

\begin{definition}
Let $\alpha$, $\beta$ and $\gamma$ be nonnnegative constants. Then,
${C}^{\beta}([a,b];X_{\alpha})$, as usual, denotes the Banach space of 
Holder continuous functions from $[a,b]$ to $X_{\alpha}$. Moreover:
\begin{enumerate}
[\rm (i)]

\item
${\mathscr C}^{\beta}([a,b]^2_{<};X_{\alpha})$ is the set of continuous functions $f:[a,b]^2_{<}\to X_{\alpha}$ such that $f(s,s)=0$ for every $s\in [a,b]$ and
\begin{align*}
\|f\|_{{\mathscr C}^{\beta}([a,b]^2_<;X_\alpha)}:=\sup_{a\le s<t\leq b}\frac{\|f(s,t)\|_{X_{\alpha}}}{(t-s)^\beta}<\infty.
\end{align*}
Similarly, ${\mathscr C}^{\beta}([a,b]^3_{<};X_{\alpha})$ is the set of continuous functions $f:[a,b]^3_{<}\to X_{\alpha}$ such that $f(r,s,t)=0$ when $r=s$ or $s=t$ and
\begin{align*}
\|f\|_{{\mathscr C}^{\beta}([a,b]^3_<;X_\alpha)}:=\sup_{a\le r<s<t\leq b}\frac{\|f(r,s,t)\|_{X_{\alpha}}}{(t-r)^\beta}<\infty.
\end{align*}
\item $\mathscr C_{-\gamma}((a,b];X_\alpha)$ is the set of continuous functions $f:(a,b]\rightarrow X_\alpha$ such that
\begin{align*}
    \|f\|_{{\mathscr C}_{-\gamma}((a,b];X_\alpha)}:=\sup_{t\in(a,b]}(t-a)^{\gamma}\|f(t)\|_{X_\alpha}<\infty.
\end{align*}
If $\gamma=0$, then we set $C_b((a,b];X_\alpha):={\mathscr C}_{0}((a,b];X_\alpha)$;
\item 
$\mathscr C_{-\gamma}^{\beta}((a,b]^2_<;X_\alpha)$ 
is the set of continuous functions $f:(a,b]^2_{<}\to X_{\alpha}$ such that $f(s,s)=0$ for every $s\in (a,b]$ and
\begin{align*}
\|f\|_{{\mathscr C}_{-\gamma}^{\beta}((a,b]^2_<;X_\alpha)}:=\sup_{a<s<t\leq b}(s-a)^{\gamma}\frac{\|f(s,t)\|_{X_{\alpha}}}{(t-s)^\beta}<\infty.
\end{align*}
Similarly,
$\mathscr C_{-\gamma}^{\beta}((a,b]^3_<;X_\alpha)$ is the set of the continuous functions $f:(a,b]^3_{<}\to X_{\alpha}$ such that $f(r,s,t)=0$ when $r=s$ or $s=t$ and
\begin{align*}
\|f\|_{{\mathscr C}_{-\gamma}^{\beta}((a,b]^3_<;X_\alpha)}:=\sup_{a<r<s<t\leq b}(r-a)^{\gamma}\frac{\|f(r,s,t)\|_{X_{\alpha}}}{(t-r)^\beta}<\infty.
\end{align*}
\end{enumerate}
\end{definition}

We introduce a new increment and an operator acting on functions $f:[a,b]^n_<\to X$.
\begin{definition}
\label{def:def_incr_comm}
For every $n\in\N$ and every function $f:[a,b]^n_<\to X$, we set:
\begin{align*}
(\delta_{S,1}f(t_1,t_2)): = & S(t_1-a)(f(t_2)-f(t_1)), \\(\delta_{S,n} f)(t_1,\ldots,t_{n+1})
:=&\sum_{i=1}^{n-1}(-1)^{n-i}f((t_1,\ldots,t_{n+1})^{\wedge i}) \\
& + S(t_{n}-t_{n-1})[f(t_1,\ldots,t_{n-1},t_{n+1})-f(t_1,\ldots,t_{n})].
\end{align*}
Further, for every $n\in\N$, $f:[a,b]_{<}^n\to X$ and $(t_1,\ldots,t_n)\in [a,b]^n_{<}$, we set
$(\mathbb{S}_1f)(t)=S(t-a)f(t)$, if $n=1$, and $(\mathbb{S}_nf)(t_1,...,t_n):=
S(t_n-t_{n-1})f(t_1,...,t_n)$, if $n\geq 2$.
\end{definition}

\begin{lemma}
\label{lemma:commutz}
For every $n\in\N$ and every function $f:[a,b]^n_<\to X$, it holds that
\begin{align}
\label{eq:commutaz}
\hat\delta_n \mathbb{S}_nf= \mathbb{S}_{n+1} \delta_{S,n} f.
\end{align}
\end{lemma}

\begin{proof}
Fix $f:[a,b]^n_<\rightarrow X$. From the definition 
of the operators $\hat\delta_n$ and $\delta_{S,n}$ it follows that
\begin{align*}
&(\hat\delta_n\mathbb S_nf)(t_1,\ldots,t_{n+1})\\
= & \sum_{i=1}^{n+1}(-1)^{n-i}(\mathbb S_nf)((t_1,\ldots,t_{n+1})^{\wedge i})-\mathfrak a(t_{n},t_{n+1})(\mathbb S_nf)(t_1,\ldots,t_n) \\
= & (-1)^{n-1}S(t_{n+1}-t_{n})f(t_2,\ldots,t_{n+1})
+(-1)^{n-2} S(t_{n+1}-t_{n})f(t_1,t_3\ldots,t_{n+1}) \\
& +\ldots+S(t_{n+1}-t_{n-1})f(t_1,\ldots,t_{n-1},t_{n+1}) -S(t_{n+1}-t_{n-1})f(t_1,\ldots,t_n) \\
= & S(t_{n+1}-t_{n})\bigg(\sum_{i=1}^{n-1}(-1)^{n-i}f((t_1,\ldots,t_{n+1})^{\wedge i}) \\
&\qquad\qquad\qquad\;\, +S(t_n-t_{n-1})[f(t_1,\ldots,t_{n-1},t_{n+1}) -f(t_1,\ldots,t_{n})]\bigg )
\end{align*}
for every $(t_1,\ldots,t_{n+1})\in [a,b]_{<}^{n+1}$ and the term in brackets is $(\delta_{S,n}f)(t_1,\ldots,t_{n+1})$.
\end{proof}

The following result  is the analogous of \cite[Proposition 3.1]{GT10}, which deals with the increment $\hat\delta_n$.

\begin{lemma}
\label{lemma:ker=im}
${\rm Im}(\delta_{S,n})={\rm Ker}(\delta_{S,n+1})$ for every $n\in\N$.
\end{lemma}
\begin{proof}
Let us begin by proving the inclusion ${\rm Im}(\delta_{S,n})\subseteq {\rm Ker}(\delta_{S,n+1})$. For this purpose, we fix a function $g\in {\rm Im}(\delta_{S,n})$ and let $f:[a,b]_{<}^n\to X$ be such that $g=\delta_{S,n}f$. From Lemma \ref{lemma:commutz} it follows that $\hat\delta_n\mathbb S_nf=\mathbb S_{n+1}g$. Clearly, the function $\hat\delta_n \mathbb S_nf$ belongs to the range of the operator $\hat\delta_n$ which, by  \cite[Proposition 3.1]{GT10}, coincides with the kernel of the operator $\hat\delta_{n+1}$. Therefore, by applying \eqref{eq:commutaz} with $n$ replaced by $n+1$ we infer that
\begin{align*}
0
=& \hat\delta_{n+1}\hat\delta_n\mathbb S_nf
= \hat\delta_{n+1}\mathbb S_{n+1}g
= \mathbb S_{n+2}\delta_{S,n+1}g.
\end{align*}
Since each operator $S(t)$ is one to one, we conclude that 
$\delta_{S,n+1}g=0$, so that $g$ belongs to the kernel of the operator $\delta_{S,n+1}$.

To prove the inclusion ``$\supset$'', we fix a function $f:[a,b]_<^{n+1}\to X$ such that $\delta_{S,n+1}f=0$ and consider the function $g:[a,b]_{<}^n\to X$, defined by  $g(t_1,\ldots,t_{n})=f(a,t_1\ldots,t_n)$ for every $(t_1,\ldots,t_n)\in [a,b]_{<}^n$. Note that
\begin{align*}
(\delta_{S,n}g)(t_1,\ldots,t_{n+1})
= & \sum_{i=1}^{n-1}(-1)^{n-i}f(a,(t_1,\ldots,t_{n+1})^{\wedge i}) \\
& +S(t_{n}-t_{n-1})[f(a,t_1,\ldots,t_{n-1},t_{n+1})-f(a,t_1,\ldots,t_{n})]\\
= & \sum_{i=1}^{n-1}(-1)^{n-i}f((a,t_1,\ldots,t_{n+1})^{\wedge i+1}) \\
& +S(t_{n}-t_{n-1})[f(a,t_1,\ldots,t_{n-1},t_{n+1})-f(a,t_1,\ldots,t_{n})]\\
= &\sum_{i=2}^n(-1)^{n+1-i}f((a,t_1,\ldots,t_{n+1})^{\wedge i}) \\
& +S(t_{n}-t_{n-1})[f(a,t_1,\ldots,t_{n-1},t_{n+1})-f(a,t_1,\ldots,t_{n})]\\
= & (\delta_{S,n+1}f)(a,t_1,\ldots,t_{n+1})-(-1)^nf(t_1,\ldots,t_{n+1}) \\
= & (-1)^{n+1}f(t_1,\ldots,t_{n+1})
\end{align*}
for every $(t_1,\ldots,t_{n+1})\in [a,b]_{<}^{n+1}$.
From this chain of equalities it follows that
$f$ is the image under the operator $\delta_{S,n}$ of the function $(-1)^{n+1}g$. The proof is complete. 
\end{proof}

\subsection{The sewing map and the convolution integral for the new increment}
\label{subsect-2.2}
The aim of this subsection is to prove the existence of a sewing map $M$ for functions in ${\rm Ker}(\delta_{S,3})$. As a byproduct, we are able to define the convolution integral for a wide class of functions.
\begin{proposition}
\label{prop:new_sew_map}
Fix $\mu>1$, $\alpha\in [0,2)$ and $f\in  \mathscr{C}^{\mu}([a,b]^3_<;X_\alpha)\cap {\rm Ker}(\delta_{S,3})$. Then, there exists a unique function $M\in\displaystyle\bigcap_{0\le\varepsilon<1} \mathscr{C}^{\mu-\varepsilon}([a,b]^2_<;X_{\alpha+\varepsilon})$ such that $\hat\delta_2 M=\mathbb S_3f$ on $[a,b]^3_{<}$. Moreover, for every $\varepsilon\in[0,1)$ there exists a positive constant $C=C(\varepsilon,\alpha,\mu,b-a)$ such that
\begin{align}
\|M\|_{\mathscr{C}^{\mu-\varepsilon}([a,b]^2_<;X_{\alpha+\varepsilon})}
\leq C\|f\|_{\mathscr{C}^{\mu}([a,b]^3_<;X_\alpha)}.
\label{croazia-canada}
\end{align}
\end{proposition}
\begin{proof}
{\it Uniqueness}. Fix $f\in \mathscr{C}^{\mu}([a,b]^3_<;X_\alpha)$ and let $M_1,M_2$ be two functions as in the statement.
Then, $\hat\delta_2 (M_1-M_2)=0$, and therefore, by  \cite[Proposition 3.1]{GT10}, it follows that there exists a function $g\in C([a,b];X_\alpha)$ such that $\hat\delta_1 g=M_1-M_2\in \mathscr{C}^{\mu}([a,b]^2_<;X_\alpha)$. From \cite[Proposition 3.4]{GT10} we conclude that $g\equiv 0$, i.e., $M_1=M_2$. 

{\it Existence}. 
Since $f\in {\rm Ker}(\delta_{S,3})$, from Lemma \ref{lemma:ker=im} it follows that there exists a function $g:[a,b]^2_<\rightarrow X$ such that $f=\delta_{S,2}g$. Lemma \ref{lemma:commutz} shows that $\mathbb S_3f=\mathbb S_3\delta_{S,2}g=\hat\delta_2\mathbb S_2g$. If we set $\psi:=\mathbb S_2g$, then we get $\hat\delta_2\psi=\mathbb S_3f$.

Since the rest of the proof is rather long, we split it into several steps. In Step 1, we construct the function $M$. Unfortunately, from the way the function $M$ is defined, is not easy to prove that $\hat\delta_2M=\mathbb S_3f$. Hence, in Steps 2 to 4, we prove that, for every $s,t\in [a,b]$, with $s<t$, it holds that
\begin{align}
\label{limite_sew_map}
M(s,t)=\psi(s,t)-\lim_{|\Pi(s,t)|\to 0}\sum_{i=1}^nS(t-t_i)\psi(t_{i-1},t_i),
\end{align}
in $X_\alpha$, where $\Pi(s,t):=\{s=t_0<t_1<\ldots<t_n=t\}$ is a partition of $[s,t]$. Using this formula, in Step 5 we complete the proof.

{\em Step 1}. Let us fix $(s,t)\in[a,b]^2_<$, with $s<t$, and for every $n\in\N\cup\{0\}$ let us consider the partition
\begin{align}
\label{def_part_diadica}
\Pi_n(s,t):=\{s=r_0^n<\ldots<r_{2^n}^n=t\}, \qquad
    r_i^n:=s+i\frac{t-s}{2^n}, \quad i=0,\ldots,2^n,   
\end{align}
of $[s,t]$.
For every $n\in\N\cup\{0\}$, we set
\begin{align}
\label{def_funz_M_n}
M_n(s,t):=\psi(s,t)-\sum_{i=1}^{2^n}S(t-r^n_i)\psi(r_{i-1}^n,r_i^n). 
\end{align}
Moreover, $M_n$ is a continuous function on $[a,b]^2_{<}\setminus\{(s,s): s\in [a,b]\}$.
Note that $r_{0}^0=s$, $r_{1}^0=t$ and, therefore, $M_0(s,t)=0$. 

By straightforward computations, we get
\begin{align*}
M_n(s,t)=\psi(s,t)-\sum_{i=1}^{2^{n}}S(t-r^{n+1}_{2i})\psi(r^{n+1}_{2i-2},r^{n+1}_{2i})
\end{align*}
and
\begin{align*}
M_{n+1}(s,t)=&\psi(s,t)-\sum_{i=1}^{2^n}S(t-r^{n+1}_{2i})\psi(r^{n+1}_{2i-1},r^{n+1}_{2i})\\
&-\sum_{i=1}^{2^n}S(t-r^{n+1}_{2i-1})\psi(r^{n+1}_{2i-2},r^{n+1}_{2i-1})\\
=&\psi(s,t)-\sum_{i=1}^{2^n}S(t-r^{n+1}_{2i})\psi(r^{n+1}_{2i-1},r^{n+1}_{2i})\\
&-\sum_{i=1}^{2^n}S(t-r^{n+1}_{2i})S(r^{n+1}_{2i}-r^{n+1}_{2i-1})\psi(r^{n+1}_{2i-2},r^{n+1}_{2i-1}),
\end{align*}
so that, recalling that $({\mathbb S}_3f)(r,s,t)=S(t-s)f(r,s,t)$, we deduce that
\begin{align*}
M_{n+1}(s,t)-M_n(s,t)= & \sum_{i=1}^{2^n}S(t-r_{2i}^{n+1})(\hat\delta_2\psi)(r_{2i-2}^{n+1}, r^{n+1}_{2i-1},r^{n+1}_{2i}) \\
= & \sum_{i=1}^{2^n}S(t-r_{2i}^{n+1})({\mathbb S}_3f)(r_{2i-2}^{n+1}, r^{n+1}_{2i-1},r^{n+1}_{2i}) \\
= & \sum_{i=1}^{2^n}S(t-r_{2i-1}^{n+1})f(r_{2i-2}^{n+1},r_{2i-1}^{n+1}, r_{2i}^{n+1}).
\end{align*}
Let us stress that, differently from the equality obtained in \cite[Theorem 3.5]{GT10}, the argument of $S(\cdot)$ is always greater than or equal to $2^{-n-1}(t-s)$, so that it never reaches $0$.
This allows to exploit the smoothing properties of the semigroup in order to get better regularity in space. Since $M_0(s,t)=0$, it follows that 
\begin{align}
\label{serie_sew_map}
M_n(s,t)=\sum_{k=0}^{n-1}[M_{k+1}(s,t)-M_k(s,t)].
\end{align}
Each term in \eqref{serie_sew_map} belongs to $X_{\alpha+\varepsilon}$ for every $\varepsilon>0$, by Hypothesis \ref{hyp-main}(iii)-(a). Moreover, for every $\varepsilon\in[0,1)$ and $k\in\N\cup\{0\}$ we can estimate
\begin{align*}
&\|M_{k+1}(s,t)-M_k(s,t)\|_{X_{\alpha+\varepsilon}}\\
\leq & L_{\alpha,\alpha+\varepsilon}\|f\|_{\mathscr{C}^\mu([a,b]_<^3;X_\alpha)}|t-s|^\mu 2^{-k\mu}\sum_{i=1}^{2^k}|t-r^{k+1}_{2i-1}|^{-\varepsilon} \\
= & L_{\alpha,\alpha+\varepsilon}\|f\|_{\mathscr{C}^{\mu}([a,b]_<^3;X_\alpha)}|t-s|^\mu 2^{-k\mu}\sum_{i=1}^{2^k}\frac{1}{r^{k+1}_{2i}-r^{k+1}_{2i-1}}\int_{r^{k+1}_{2i-1}}^{r^{k+1}_{2i}}|t-r^{k+1}_{2i-1}|^{-\varepsilon}d\xi \\
\leq & 2L_{\alpha,\alpha+\varepsilon}\|f\|_{\mathscr{C}^{\mu}([a,b]_<^3;X_\alpha)}|t-s|^{\mu-1} 2^{k(1-\mu)}\int_s^t(t-\xi)^{-\varepsilon}d\xi \\
= & 2^{1+k(1-\mu)}(1-\varepsilon)^{-1}L_{\alpha,\alpha+\varepsilon}\|f\|_{\mathscr{C}^{\mu}([a,b]_<^3;X_\alpha)}|t-s|^{\mu-\varepsilon},
\end{align*}
where $L_{\alpha,\alpha+\varepsilon}=L_{\alpha,\alpha+\varepsilon,b-a}$ and we have used the fact that the function $\xi\mapsto (t-\xi)^{-\varepsilon}$ is increasing in $(s,t)$. It follows that the series defined in \eqref{serie_sew_map} converges in $X_{\alpha+\varepsilon}$ as $n$ tends to $\infty$. 

Denote by $M(s,t)$ the limit in $X_{\alpha+\varepsilon}$ of the sequence $\{M_n\}$. Since this 
sequence uniformly converges in $[a,b]^2_{<}$, the function $M$ is continuous in $[a,b]^2_{<}\setminus\{(s,s): s\in [a,b]\}$. Using the above computations,  we can easily extend $M$ on the diagonal of  $[a,b]^2_{<}$, setting $M(s,s)=0$ for every $s\in [a,b]$. The so obtained function, which we still denote by $M$, belongs to $\mathscr{C}^{\mu-\varepsilon}([a,b]^2_{<};X_{\alpha+\varepsilon})$ for every $\varepsilon\in [0,1)$ and formula \eqref{croazia-canada} holds true.

{\it Step 2}. This is the crucial step to prove formula \eqref{limite_sew_map}.
Let us fix $s',t'\in [a,b]$ with $s'<t'$, and let us consider a partition $\Pi^{(n)}(s',t')=\{s'=r_0^n <r^n_1<\ldots<r^n_n=t'\}$ of $[s',t']$, with more than two points. We set
\begin{align}
\widehat M_n(s',t')= \psi(s',t')-\sum_{i=1}^nS(t'-r^n_i)\psi(r^n_{i-1},r^n_{i}).
\label{def-M}
\end{align}
and notice that there exists $i\in\{1,\ldots,n-1\}$ such that $r_{i+1}^n-r_{i-1}^n\leq \frac{2(t'-s')}{n-1}$. Indeed, suppose by contradiction that $r_{i+1}^n-r_{i-1}^n> \frac{2(t'-s')}{n-1}$ for every $i=1,\ldots,n-1$. Then, from the formula
\begin{align*}
t'-s'=\sum_{i=1}^{n-1}(r^n_{i+1}-r^n_{i-1})-r_{n-1}^n+r_1^n,    
\end{align*}
we get the contradiction $t'-s'>2(t'-s')-(r_{n-1}^n-r_1^n)>t'-s'$. 

Let us denote by $\overline i_n$ an element of $\{1,\ldots,n-1\}$ such that $r^n_{i+1}-r^n_{i-1}\leq \frac{2(t'-s')}{n-1}$ and consider the partition 
$\Pi^{(n-1)}(s',t')=\{s'=r^{n-1}_0<r^{n-1}_1<\ldots<r^{n-1}_{n-1}=t'\}\subset\Pi^{(n)}(s',t')$ of $[s',t']$, where
$r_i^{n-1}=r_i^{n}$ for every $i\in\{0,\ldots,\overline i_n-1\}$, and $r_{i-1}^{n-1}=r_{i}^{n}$ for every $i\in \{\overline i_n+1,\ldots,n\}$. 
Accordingly to \eqref{def-M}, we define 
\begin{align*}
\widehat M_{n-1}(s',t')=\psi(s',t')-\sum_{i=1}^{n-1}S(t'-r^{n-1}_i)\psi(r^{n-1}_{i-1},r^{n-1}_{i}).     
\end{align*}
Arguing as above, we infer that there exists $\overline i_{n-1}\in\{1,\ldots,n-2\}$ such that
$r^{n-1}_{\overline i_{n-1}+1}-r^{n-1}_{\overline i_{n-1}-1}\leq \frac{2(t'-s')}{n-2}$. We set $r_i^{n-2}:=r_i^{n-1}$ for every $i\in\{0,\ldots,\overline i_{n-1}-1\}$ and $r_{i-1}^{n-2}:=r_{i}^{n-1}$ for every $i\in\{\overline i_{n-1}+1,\ldots,n-1\}$. Iterating this procedure, for every $k\in\{1,\ldots,n-1\}$ we define a partition $\Pi^{(k)}(s',t')=\{s'=r_0^k<r_1^k<\ldots<r_k^k=t'\}\subset \Pi^{(k+1)}(s',t')$ of $[s',t']$ and a function 
\begin{align*}
\widehat M_k(s',t')=\psi(s',t')-\sum_{i=1}^kS(t'-r^k_i)\psi(r_{i-1}^k,r_i^k).
\end{align*}
Fix $k=1,\ldots,n-1$ and denote by $\overline{i}_{k+1}$ the index such that $\Pi^{(k)}(s',t')=\{
s'=r^{k+1}_0<\ldots<r^{k+1}_{\overline{i}_{k+1}-1}<r^{k+1}_{\overline{i}_{k+1}+1}<\ldots<r^{k+1}_{k+1}\}$. Note that
\begin{align*}
\widehat M_{k+1}(s',t')=&\psi(s',t')-\sum_{i=1}^{k+1}S(t'-r^{k+1}_i)\psi(r^{k+1}_{i-1},r^{k+1}_i)\\
=&\psi(s',t')-\sum_{i=1}^{\overline{i}_{k+1}-1}S(t'-r^{k+1}_i)\psi(r^{k+1}_{i-1},r^{k+1}_i)\\
&-
S(t'-r^{k+1}_{\overline{i}_{k+1}})\psi(r^{k+1}_{\overline{i}_{k+1}-1},r^{k+1}_{\overline{i}_{k+1}})-S(t'-r^{k+1}_{\overline{i}_{k+1}+1})\psi(r^{k+1}_{\overline{i}_{k+1}},r^{k+1}_{\overline{i}_{k+1}+1})\\
&-\sum_{\overline{i}_{k+1}+2}^{k+1}S(t'-r^{k+1}_i)\psi(r^{k+1}_{i-1},r^{k+1}_i)\\
=&\psi(s',t')-\sum_{i=1}^{\overline{i}_{k+1}-1}S(t'-r^k_i)\psi(r^k_{i-1},r^k_i)-\sum_{\overline{i}_{k+1}+1}^kS(t'-r^k_i)\psi(r^k_{i-1},r^k_i)\\
&-S(t'-r^{k+1}_{\overline{i}_{k+1}})\psi(r^{k+1}_{\overline{i}_{k+1}-1},r^{k+1}_{\overline{i}_{k+1}})-S(t'-r^{k+1}_{\overline{i}_{k+1}+1})\psi(r^{k+1}_{\overline{i}_{k+1}},r^{k+1}_{\overline{i}_{k+1}+1})
\end{align*}
so that
\begin{align*}
\widehat M_{k+1}(s',t')-\widehat M_k(s',t')    = 
& S(t'-r^{k+1}_{\overline i_{k+1}+1})\psi(r^{k+1}_{\overline i_{k+1}-1},r^{k+1}_{\overline i_{k+1}+1}) \\
& - S(t'-r^{k+1}_{\overline i_{k+1}+1})\psi(r^{k+1}_{\overline i_{k+1}},r^{k+1}_{\overline i_{k+1}+1}) \\
& -S(t'-r^{k+1}_{\overline i_{k+1}})\psi(r^{k+1}_{\overline i_{k+1}-1},r^{k+1}_{\overline i_{k+1}}) \\
= & S(t'-r^{k+1}_{\overline i_{k+1}+1})(\hat\delta_2\psi)(r^{k+1}_{\overline i_{k+1}-1},r^{k+1}_{\overline i_{k+1}},r^{k+1}_{\overline i_{k+1}+1}) \\
= & S(t'-r^{k+1}_{\overline i_{k+1}})f(r^{k+1}_{\overline i_{k+1}-1},r^{k+1}_{\overline i_{k+1}},r^{k+1}_{\overline i_{k+1}+1})
\end{align*}
for every $k=1,\ldots,n-1$.
Assumption \eqref{stime_smgr}(a) and the hypothesis on $f$ imply that
\begin{align}
\|\widehat M_{k+1}(s',t')-\widehat M_{k}(s',t')\|_{X_\alpha} \notag \leq & L_{\alpha,\alpha}\|f\|_{\mathscr{C}^{\mu}([a,b]_<^3;X_\alpha)}|r^{k+1}_{\overline i_{k+1}+1}-r^{k+1}_{\overline i_{k+1}-1}|^\mu \\
\leq & 2^{\mu}L_{\alpha,\alpha}(t'-s')^{\mu} \|f\|_{\mathscr{C}^\mu([a,b]_<^3;X_\alpha)}k^{-\mu}
\label{stima_widetilde_Mk}
\end{align}
for every $k=1,\ldots,n-1$. 
From $\widehat M_1(s',t')=0$ it follows that
\begin{align*}
\widehat M_n(s',t')
= \sum_{k=1}^{n-1} \widehat [M_{k+1}(s',t')-\widehat M_k(s',t')],
\end{align*}
and from \eqref{stima_widetilde_Mk} we infer that
\begin{align}
\|\widehat M_n\|_{X_\alpha}
\le 2^{\mu}(t's')^{\mu}L_{\alpha,\alpha}\|f\|_{\mathscr{C}^\mu([a,b]_<^3;X_\alpha)}\sum_{k=1}^{n-1}k^{-\mu}\le C_{\mu}\|f\|_{\mathscr{C}^\mu([a,b]_<^3;X_\alpha)}(t'-s')^\mu,
\label{stima_serie_s't'}
\end{align}
where $C_\mu:=2^{\mu}L_{\alpha,\alpha}\sum_{k=1}^\infty k^{-\mu}$. 

{\it Step 3}. Let us fix $s,t\in[a,b]$, with $s<t$, and let $\Pi_1(s,t)=\{s=u_0<u_1<\ldots<u_m=t\}$ and $\Pi_2(s,t)=\{s=w_0<w_1<\ldots<w_h=t\}$ be two partitions of $[s,t]$ with $\Pi_1(s,t)\subset \Pi_2(s,t)$. For every $i=1,\ldots,m$, let us denote by $s^i_j$, $j=0,\ldots, j_i$, the elements of $ \Pi_2(s,t)$ which satisfy $u_{i-1}=s_0^i<s^i_1<\ldots<s^i_{j_i}=u_{i}$. If we set
\begin{align*}
M_{\Pi_1}(s,t):= & \psi(s,t)-\sum_{i=1}^m S(t-u_i)\psi(u_{i-1},u_i), \\
M_{\Pi_2}(s,t)
:= & \psi(s,t)-\sum_{j=1}^h S(t-w_j)\psi(w_{j-1},w_j) \\
= & \psi(s,t)-\sum_{i=1}^m\sum_{j=0}^{j_i}S(t-s^i_{j})\psi(s^i_{j-1},s^i_{j}),
\end{align*}
then we get
\begin{align*}
M_{\Pi_1}(s,t)-M_{\Pi_2}(s,t)
= & \sum_{i=1}^m\bigg(S(t-u_i)\psi(u_{i-1},u_i)-\sum_{j=0}^{j_i}S(t-s^i_{j})\psi(s^i_{j-1},s^i_j) \bigg ) \\
= & \sum_{i=1}^mS(t-u_i)\bigg (\psi(u_{i-1},u_i)-\sum_{j=0}^{j_i}S(u_i-s^i_{j})\psi(s^i_{j-1},s^i_j) \bigg ).
\end{align*}
Thanks to \eqref{stima_serie_s't'} we can estimate every term of the above sum by setting $s'=u_{i-1}$ and $t'=u_i$ for every $i=1,\ldots,m$. It follows that
\begin{align}
\notag
\|M_{\Pi_1}(s,t)- M_{\Pi_2}(s,t)\|_{X_\alpha}
\leq & L_{\alpha,\alpha}C_\mu\|f\|_{\mathscr{C}^\mu([a,b]_<^3;X_\alpha)}\sum_{i=1}^m(u_{i}-u_{i-1})^\mu \\
\leq & L_{\alpha,\alpha}C_\mu\|f\|_{\mathscr{C}^\mu([a,b]_<^3;X_\alpha)}|t-s|\max_{i=1,\ldots,m}(u_{i}-u_{i-1})^{\mu-1}.
\label{estimate_step_2}
\end{align} 

{\it Step 4}. Now, we are ready to prove \eqref{limite_sew_map}. We fix $s,t\in [a,b]$, with $s<t$, and $\varepsilon>0$, and we prove that, if we choose $\delta:=[(4L_{\alpha,\alpha}C_\mu\|f\|_{{\mathscr C}^{\mu}([a,b]^3_{<};X_{\alpha})}(t-s))^{-1}\varepsilon]\wedge 1$, then for every partition $\Pi(s,t)=\{s=t_0<t_1<\ldots<t_n=t\}$ of $[s,t]$, with $|\Pi(s,t)|\leq \delta$, we get
\begin{align}
\label{ineq_conv_sew_map}
\bigg\|M(s,t)-\psi(s,t)+\sum_{i=1}^nS(t-t_i)\psi(t_{i-1},t_i)\bigg\|_{X_\alpha}\leq \varepsilon.
\end{align}
This will yield \eqref{limite_sew_map}.

We fix a partition $\Pi(s,t)$ of $[s,t]$ as above and set 
\begin{align*}
M_{\Pi}(s,t)=\psi(s,t)-\sum_{i=1}^nS(t-t_i)\psi(t_{i-1},t_i).
\end{align*}
We recall that there exists $\overline m\in\N$ such that, for every $m\geq \overline m$, it holds that $\|M(s,t)- M_m(s,t)\|_{X_\alpha}\leq \varepsilon/2$, where $M_m(s,t)$ has been defined in \eqref{def_funz_M_n} for every $m\in\N\cup\{0\}$. Without loss of generality, we may assume $\overline m\geq \log_2(\delta^{-1}(t-s))$, which implies that $\frac{t-s}{2^{\overline m}}\leq \delta$. Therefore,
\begin{align}
\notag
\|M(s,t)-M_{\Pi}(s,t)\|_{X_\alpha}
\leq & \|M(s,t)-M_{\overline m}(s,t)\|_{X_\alpha}+\|M_{\overline m}(s,t)- M_{\Pi}(s,t)\|_{X_\alpha} \\
\leq & \frac{\varepsilon}{2}+\|M_{\overline m}(s,t)- M_{\Pi}(s,t)\|_{X_\alpha}.
\label{estimate_somma_step3}
\end{align}
We set $\widehat \Pi(s,t):=\Pi_{\overline m}(s,t)\cup\Pi(s,t)=:\{s_0=s<s_1<\ldots<s_h=t\}$ for some $h\in\{\max\{n,2^{{\overline m}}\},\ldots, n+2^{{\overline m}}-1\}$, where $\Pi_m(s,t)$ has been introduced in \eqref{def_part_diadica} for every $m\in\N\cup\{0\}$. We also set
\begin{align*}
M_{\widehat \Pi}(s,t):=\psi(s,t)-\sum_{i=1}^hS(t-s_i)\psi(s_{i-1},s_i).   
\end{align*}
Since both $\Pi(s,t)$ and $\Pi_{\overline m}(s,t)$ are contained in $\widehat \Pi(s,t)$, from estimate \eqref{estimate_step_2} we infer that
\begin{align*}
\|M_{\overline m}(s,t)-M_{\Pi}(s,t)\|_{X_\alpha}
\leq & \|M_{\overline m}(s,t)- M_{\widehat\Pi}(s,t)\|_{X_\alpha}+\|M_{\widehat\Pi}(s,t)-M_{\Pi}(s,t)\|_{X_\alpha} \\
\leq & L_{\alpha,\alpha}C_\mu
\|f\|_{{\mathscr C}^{\mu}([a,b]^3_{<};X_{\alpha})}(|\Pi_{\overline m}(s,t)|+|\Pi(s,t)|)(t-s) \\
\leq & 2L_{\alpha,\alpha}C_\mu\delta\|f\|_{{\mathscr C}^{\mu}([a,b]^3_{<};X_{\alpha})}(t-s)\leq \frac\varepsilon2,
\end{align*}
which gives \eqref{ineq_conv_sew_map} combined  with \eqref{estimate_somma_step3}.

{\em Step 5.} Now, we complete the proof, using formula \eqref{limite_sew_map} to show that
\begin{eqnarray*}
(\hat\delta_2 M)(r,s,t)
=(\hat\delta_2\psi)(r,s,t)=(\mathbb{S}_3f)(r,s,t),\qquad\;\, (r,s,t)\in[a,b]^3_<.
\end{eqnarray*}

We first observe that
\begin{align*}
(\hat\delta_2 M)(r,r,t)=M(r,t)-M(r,t)-S(t-r)M(r,r)=0
\end{align*}
and $(\mathbb S_3f)(r,r,t)=S(t-r)f(r,r,t)=0$ since, by definition of $\mathscr{C}^{\mu}([a,b]^3_{<};X_{\alpha})$, $f$ vanishes at the points of $[a,b]^3_{<}$ with at least two components which coincide.
Hence, $(\hat\delta_2 M)(r,r,t)
=(\mathbb S_3f)(r,r,t)$. In the same way, we can show that $(\hat\delta_2 M)(r,s,s)
=(\mathbb S_3f)(r,s,s)$ for every $a\le r<s\le b$. 

Let us consider the case when $a\le r<s<t\le b$. For this purpose, we
use \eqref{limite_sew_map} to show that, for every $a\le r<s<t\le b$, it holds that 
\begin{align*}
(\hat\delta_2M)(r,s,t)
= & M(r,t)-M(s,t)-S(t-s)M(r,s) \\
= & \psi(r,t)-\lim_{|\Pi(r,t)|\rightarrow0}\sum_{i=1}^nS(t-t_i)\psi(t_{i-1},t_i) \\
& -\psi(s,t)+\lim_{|\Pi(s,t)|\rightarrow0}\sum_{i=1}^nS(t-t_i)\psi(t_{i-1},t_i) \\
&-S(t-s)\psi(r,s)+S(t-s)\lim_{|\Pi(r,s)|\rightarrow0}\sum_{i=1}^nS(s-t_i)\psi(t_{i-1},t_i) \\
= & (\hat\delta_2\psi)(r,s,t)-\lim_{|\Pi(r,t)|\rightarrow0}\sum_{i=1}^nS(t-t_i)\psi(t_{i-1},t_i) \\
& +\lim_{|\Pi(s,t)|\rightarrow0}\sum_{i=1}^nS(t-t_i)\psi(t_{i-1},t_i) +\lim_{|\Pi(r,s)|\rightarrow0}\sum_{i=1}^nS(t-t_i)\psi(t_{i-1},t_i).
\end{align*}
Since
\begin{align*}
& \lim_{|\Pi(r,t)|\rightarrow0}\sum_{i=1}^nS(t-t_i)\psi(t_{i-1},t_i) \\
= & \lim_{|\Pi(s,t)|\rightarrow0}\sum_{i=1}^nS(t-t_i)\psi(t_{i-1},t_i) +\lim_{|\Pi(r,s)|\rightarrow0}\sum_{i=1}^nS(t-t_i)\psi(t_{i-1},t_i),
\end{align*}
the assertion follows at once.
\end{proof}
Let us provide an example of a function $f$ which satisfies the assumptions of Proposition \ref{prop:new_sew_map}. The great relevance of this example will be made clear in the second part of Subsection \ref{subsect-2.3}.

\begin{example}
\label{ex:new_sew_map}
Fix two positive numbers $\rho$ and $\eta$ such $\eta+\rho>1$. Further, let $x\in C^\eta([a,b])$ and $\varphi:[a,b]\rightarrow X$ be such that $\hat\delta_{1}\varphi\in \mathscr C^{\rho}([a,b]^2_<;X_\alpha)$ for some $\alpha\in [0,2)$. Finally, let $g:[a,b]^2_{<}\to X_{\alpha}$ be the function defined by $g(s,t):=(x(t)-x(s))\varphi(s)$ for every $(s,t)\in[a,b]^2_<$. Note that 
\begin{align}
(\delta_{S,2}g)(r,s,t)
= & -g(s,t)+S(s-r)g(r,t)-S(s-r)g(r,s)\notag \\
= & (x(t)-x(s))(-\varphi(s)+S(s-r)\varphi(r))\notag \\
= & -(x(t)-x(s))(\hat\delta_1\varphi)(r,s)
\label{pixel}
\end{align}
for every $(r,s,t)\in[a,b]^3_<$.
This means that the function $f=\delta_{S,2} g$ belongs to $\mathscr{C}^{\mu}([a,b]^3_<;X_\alpha)$ with $\mu=\eta+\rho$. Further, from Lemma $\ref{lemma:ker=im}$ we infer that $\delta_{S,3}f=0$. The assumptions of Proposition $\ref{prop:new_sew_map}$ are satisfied and, consequently, there exists a unique function $M\in\displaystyle\bigcap_{\varepsilon\in [0,1)}\mathscr{C}^{\mu-\varepsilon}([a,b]^2_<;X_{\alpha+\varepsilon})$ such that $\hat\delta_2M=\mathbb{S}_3f=\mathbb{S}_3\delta_{S,2}g=\hat\delta_2\mathbb{S}_2g$. Moreover, 
\begin{align*}
\|M\|_{\mathscr{C}^{\mu-\varepsilon}([a,b]^2_<;X_{\alpha+\varepsilon})}
\leq C\|x\|_{C^\eta([a,b]}\|\hat\delta_1\varphi\|_{\mathscr{C}^\rho([a,b]^2_<;X_\alpha)}
\end{align*}
for every $\varepsilon\in[0,1)$ and some positive constant $C=C(\varepsilon,\alpha,\mu)$.
\end{example}

\medskip

Let $g\in C([a,b]_<^2;X_\alpha)$ be a function such $\delta_{S,2}g\in \mathscr{C}^\mu([a,b]^3_<;X_\alpha)$ for some $\alpha\in [0,2)$ and $\eta>1$. Following \cite{GT10}, we introduce the function $k_g:[a,b]^2_{<}\to X$ defined by 
\begin{equation}
k_g(s,t)=S(t-s)g(s,t)-M(s,t),\qquad\;\, (s,t)\in [a,b]^2_{<},
\label{funct-k}
\end{equation}
where $M$ is the function defined in Proposition \ref{prop:new_sew_map}, associated to the function $f=\delta_{S,2} g$. Using the arguments in the last part of the proof of Proposition \ref{prop:new_sew_map}, it can be easily checked that the function ${\mathscr I}_g=k_g(a,\cdot)$ satisfies the condition $(\hat\delta_1{\mathscr I}_g)(s,t)=k_g(s,t)$
for every $(s,t)\in [a,b]^2_{<}$ and belongs to $C([a,b];X_{\alpha})$. Moreover, ${\mathscr I}_g$ vanishes at $t=a$ and this is the unique function with this property which belongs to $C([a,b];X)$. Indeed, suppose that ${\mathscr J}$ is another function in $C([a,b];X)$ which vanishes at $t=a$ and satisfies the condition $\hat\delta_1{\mathscr J}=k_g$. Then, the function $h={\mathscr I}_g-\mathscr{J}$ vanishes at $a$ and $\hat\delta_1h=0$ in $[a,b]$. In particular, $(\hat\delta_1h)(a,t)=0$ for every $t\in [a,b]$, which means that
$h(t)-S(t-a)h(a)=h(t)$ vanishes for every $t\in [a,b]$.

Inspired by \cite{GT10}, we provide the following definition.

\begin{definition}
\label{def-2.8}
Let $g\in C([a,b]_<^2;X_\alpha)$, for some $\alpha\in [0,2)$, be such that $\delta_{S,2}g\in \mathscr{C}^{\mu}([a,b]^3_<;X_\alpha)$ for some $\mu>1$. The function ${\mathscr I}_g=k_g(a,\cdot)$, where $k_g$ has been defined in \eqref{funct-k}, is called {\textbf{convolution integral of $\bm{g}$}}. 
\end{definition}

\subsection{Convolution integrals with singularities}
\label{subsect-2.3}
Since the map $t\mapsto \|S(t)\psi\|_{X_\alpha}$ has a singularity at $t=0$ of order $\alpha-\rho$ if $\psi\in X_\rho$, we need to extend the statement of Proposition \ref{prop:new_sew_map} when $f\in \mathscr{C}^{\mu}_{-\gamma}((a,b]^3_<;X_\alpha)$, i.e., when the function $f$ has a singularity of order $\gamma>0$ at $s=a$, and $\mu>1$. To begin with, we show that, thanks to Proposition \ref{prop:new_sew_map}, we can define a unique function $M$ on $(a,b]_<^2$ which enjoys nice properties.  

\begin{lemma}
\label{lem:function_M}
Fix  $\mu>1$, $\alpha\in [0,2)$ and let $g:(a,b]_<^2\rightarrow X$  be a function such that $\delta_{S,2}g$ belongs to $\mathscr{C}
^{\mu}_{-\gamma}((a,b]_<^3;X_{\alpha})$. Then, there exists a unique function $M$ which belongs to $\displaystyle \bigcap_{\varepsilon\in [0,1)}\mathscr{C}^{\mu-\varepsilon}([a+\theta,b]^2_<;X_{\alpha+\varepsilon})$ for every $\theta\in(0,b-a)$ and satisfies the condition $\hat\delta_2 M=\mathbb S_3\delta_{S,2}g$ in $(a,b]^3_<$.
\end{lemma}
\begin{proof}
We begin the proof by observing that the function $f=\delta_{S,2}g$ satisfies the assumptions of Proposition \ref{prop:new_sew_map} in $[a+\theta,b]^3_<$ for every $\theta\in(0,b-a)$. Hence, for every $\theta\in(0,b-a)$ there exists a unique function $M_{\theta}\in \displaystyle\bigcap_{\varepsilon\in[0,1)}
\mathscr{C}^{\mu-\varepsilon}([a+\theta,b]^2_<;X_{\alpha+\varepsilon})$ such that $\hat\delta_2 M_{\alpha}=\mathbb S_3f$ in $[a+\theta,b]^3_<$.

Note that, if $0<\theta_1<\theta_2<b-a$, then $M_{\theta_1}$ and $M_{\theta_2}$ coincide on $[a+\theta_2,b]^2_<$. Therefore, if for every $(s,t)\in(a,b]^2_<$ we set $M(s,t)=M_{\theta}(s,t)$ for some $\theta\in(0,s-a)$, then the function $M$ fulfills the required properties.
\end{proof}

\begin{theorem}
\label{thm:sew_map_sing_int}
Fix $\eta\in(0,1)$, $\mu>1$, $\alpha,\beta\in [0,2)$, with $0\leq \alpha-\beta<1$, $\gamma\in(0,\eta\wedge(\mu+\beta-\alpha))$ and suppose that $g$ belongs to $\mathscr{C}^{\eta}_{-\gamma}((a,b]^2_<;X_\alpha)$ and satisfies the condition $\delta_{S,2}g\in \mathscr{C}^{\mu}_{-\gamma}((a,b]_<^3;X_\beta)$. Then, the function $k_g:(a,b]^2_<\to X$, defined by $k_g(s,t):=S(t-s)g(s,t)-M(s,t)$ for every $(s,t)\in (a,b]^2_{<}$ $($where $M$ is the function defined in Lemma $\ref{lem:function_M})$ can be extended up to $s=a$ and it belongs to $\displaystyle\bigcap_{\varepsilon\in[0,\varepsilon_0)}\mathscr{C}^{\eta\wedge(\mu+\beta-\alpha)-\gamma-\varepsilon}([a,b]^2_<;X_{\alpha+\varepsilon})$, where
$\varepsilon_0=(1+\beta-\alpha)\wedge(\eta\wedge (\mu+\beta-\alpha)-\gamma)$.
Further, $\hat\delta_2 k_g=0$ on $[a,b]^3_<$ and there exists a positive constant $C$, which depends on $\varepsilon$, $\mu$, $b-a$, $\eta$, $\gamma$, $\alpha$ and $\beta$, such that
\begin{align}
\label{stima_sew_map_sing}
\|k_g\|_{\mathscr{C}^{\eta\wedge(\mu+\beta-\alpha)-\gamma-\varepsilon}([a,b]^2_<;X_{\alpha+\varepsilon})}\leq C(\|g\|_{\mathscr{C}^{\eta}_{-\gamma}((a,b]^2_<;X_\alpha)}+\|\delta_{S,2}g\|_{\mathscr{C}^{\mu}_{-\gamma}((a,b]^3_<;X_\beta)}).
\end{align}
\end{theorem}

\begin{proof}
To begin with, we observe that the function $g$ satisfies the assumptions of Lemma \ref{lem:function_M}, with $\alpha$ being replaced by $\beta$.
Hence, there exists a unique function $M$, which belongs to $\displaystyle\bigcap_{\varepsilon\in [0,1)}\mathscr{C
}^{\mu-\varepsilon}([a+\theta,b]^2_{<};X_{\beta+\varepsilon})$ for every $\theta\in (0,b-a)$, such that $\hat\delta_2M=\mathbb{S}_3\delta_{S,2}g=\hat\delta_2{\mathbb S}_2g$ in $(a,b]^3_{<}$ (see Lemma \ref{lemma:commutz}). 

Let us fix $\varepsilon\in [0,\varepsilon_0)$, $(s,t)\in(a,b]^2_<$, with $s<t$, and $n\in\N$, and introduce the function
\begin{align}
\widetilde M_n(s,t)=\psi(s,t)-\sum_{i=2}^{2^n}S(t-r_i^n)\psi(r_{i-1}^n,r^n_i)
=M_n(s,t)+S(t-r^n_1)\psi(s,r^n_1),
\label{sfumature}
\end{align}
where $\psi=\mathbb S_2g$ for every $(s_1,s_2)\in (a,b]^2_{<}$,
$r_i^n=s+\frac{i}{2^n}(t-s)$ for every $n\in\N$ and $i=0,\ldots,2^n$, and the function $M_n$ is defined in 
\eqref{def_funz_M_n}.
We omit from the definition of $\widetilde M$, the term 
$S(t-r^n_1)\psi(s,r^{n+1}_1)$ in order to stay away from the singularity at $s=a$.
For every $n\geq 1$ we get
\begin{align*}
& \widetilde M_{n+1}(s,t)-\widetilde M_n(s,t) \\
=& M_{n+1}(s,t)-M_n(s,t)+S(t-r^{n+1}_1)\psi(s,r^{n+1}_1)
-S(t-r^n_1)\psi(s,r^n_1)\\
=& \sum_{i=2}^{2^n}S(t-r_{2i-1}^{n+1})(\delta_{S,2}g)(r_{2i-2}^{n+1},r^{n+1}_{2i-1},r^{n+1}_{2i}) \\
&-S(t-r^{n+1}_1)[(\delta_{S,2}g)(s,r^{n+1}_1,r^{n+1}_2)-\psi(s,r^{n+1})]-S(t,r^{n+1}_2)\psi(s,r^{n+1}_2)\\
= & \sum_{i=2}^{2^n}S(t-r_{2i-1}^{n+1})(\delta_{S,2}g)(r_{2i-2}^{n+1},r^{n+1}_{2i-1},r^{n+1}_{2i}) -S(t-r^{n+1}_1)g(r^{n+1}_1,r^{n+1}_2).
\end{align*}
Hence,
\begin{align}
& \|\widetilde M_{n+1}(s,t)-\widetilde M_n(s,t)\|_{X_{\alpha+\varepsilon}} \notag \\
\leq &  L_{\beta,\alpha+\varepsilon}\|\delta_{S,2}g\|_{\mathscr{C}^{\mu}_{-\gamma}((a,b]^3_<;X_\beta)}\sum_{i=2}^{2^n}|r_{2i}^{n+1}-r^{n+1}_{2i-2}|^{\mu}|t-r_{2i-1}^{n+1}|^{\beta-\alpha-\varepsilon}|r_{2i-2}^{n+1}-a|^{-\gamma} \notag  \\
& + L_{\alpha,\alpha+\varepsilon}\|g\|_{\mathscr{C}^{\eta}_{-\gamma}((a,b]^2_<;X_\alpha)}|t-r^{n+1}_1|^{-\varepsilon}2^{-(n+1)\eta}(t-s)^\eta|r_1^{n+1}-a|^{-\gamma} \notag \\
\leq & L_{\beta,\alpha+\varepsilon} 
\|\delta_{S,2}g\|_{\mathscr{C}^{\mu}_{-\gamma}((a,b]^3_<;X_\beta)}|t-s|^\mu2^{-n\mu}
\sum_{i=2}^{2^n}|t-r_{2i-1}^{n+1}|^{\beta-\alpha-\varepsilon}|r_{2i-2}^{n+1}-a|^{-\gamma} \notag \\
& +2^{-\eta+\gamma+\varepsilon}2^{n(\gamma-\eta)}L_{\alpha,\alpha+\varepsilon}\|g\|_{\mathscr{C}^{\eta}_{-\gamma}((a,b]^2_<;X_\alpha)}|t-s|^{\eta-\gamma-\varepsilon}, \label{stima_M_tilde}
\end{align}
where the constants $L_{\rho,\delta}=L_{\rho,\delta,b-a}$ have been defined in Hypothesis \ref{hyp-main}$(iii)$-$(a)$ and in the last step of \eqref{stima_M_tilde}, we have used the inequality
\begin{align*}
|t-r^{n+1}_1|^{-\varepsilon}|r^{n+1}_1-a|^{-\gamma}
=&(t-s)^{-\varepsilon}(1-2^{-n-1})^{-\varepsilon}\bigg (s-a+\frac{t-s}{2^{n+1}}\bigg )^{-\gamma}\\
\le &2^{\varepsilon}(t-s)^{-\varepsilon}
\bigg (\frac{t-s}{2^{n+1}}\bigg )^{-\gamma}
=2^{\varepsilon+\gamma(1+n)}(t-s)^{-\varepsilon-\gamma}.
\end{align*}

Let us estimate the first term in the right-hand side of \eqref{stima_M_tilde}. We stress that, differently from the computations in the proof of Proposition \ref{prop:new_sew_map}, we have the additional factor $|r^{n+1}_{2i-2}-a|^{-\gamma}$, which arises from the singularity of $\delta_{S,2}g$ at $a$. Note that
\begin{align*}
r_{2i-2}^{n+1}-a=r_{2i}^{n+1}-\frac{t-s}{2^n}-a\geq \frac12(r_{2i}^{n+1}-a) 
\end{align*}
if $i\geq 2$. Hence,
\begin{align*}
& \sum_{i=2}^{2^n}|t-r_{2i-1}^{n+1}|^{\beta-\alpha-\varepsilon}|r_{2i-2}^{n+1}-a|^{-\gamma} \\
\leq & \sum_{i=2}^{2^n}\frac{2^{\gamma}}{r_{2i}^{n+1}-r^{n+1}_{2i-1}}\int_{r_{2i-1}^{n+1}}^{r^{n+1}_{2i}}|t-r_{2i-1}^{n+1}|^{\beta-\alpha-\varepsilon}|r_{2i}^{n+1}-a|^{-\gamma}d\xi \\
\leq & 2^{\gamma+1} 2^{n}|t-s|^{-1}\int_{s}^t(t-\xi)^{\beta-\alpha-\varepsilon}(\xi-a)^{-\gamma}d\xi,
\end{align*}
where we have used the fact that the function $\xi\mapsto (t-\xi)^{\beta-\alpha-\varepsilon}$ is increasing in $(-\infty,t)$ and the function $\xi\mapsto (\xi-a)^{-\gamma}$ is decreasing in $(a,\infty)$. Further,
\begin{align*}
\int_{s}^t(t-\xi)^{\beta-\alpha-\varepsilon}(\xi-a)^{-\gamma}d\xi
\leq & \int_{s}^t(t-\xi)^{\beta-\alpha-\varepsilon}(\xi-s)^{-\gamma}d\xi \\
= &B(\beta-\alpha-\varepsilon+1,1-\gamma)(t-s)^{1+\beta-\alpha-\gamma-\varepsilon}.
\end{align*}
It follows that
\begin{align*}
& \|\widetilde M_{n+1}(s,t)-\widetilde M_n(s,t)\|_{X_{\alpha+\varepsilon}} \notag \\
\le &2^{\gamma+1+n(1-\mu)}L_{\beta,\alpha+\varepsilon}B(\beta-\alpha-\varepsilon+1,1-\gamma)\|\delta_{S,2}g\|_{\mathscr{C}^{\mu}_{-\gamma}((a,b]^3_<;X_\beta)}|t-s|^{\mu+\beta-\alpha-\gamma-\varepsilon}\\
& +2^{-\eta+\gamma+\varepsilon}2^{n(\gamma-\eta)}L_{\alpha,\alpha+\varepsilon}\|g\|_{\mathscr{C}^{\eta}_{-\gamma}((a,b]^2_<;X_\alpha)}|t-s|^{\eta
-\gamma-\varepsilon}. 
\end{align*}

Next, we note that
\begin{align*}
\widetilde M_n(s,t)-\psi(s,t)
= \sum_{k=1}^{n-1}[\widetilde M_{k+1}(s,t)-\widetilde M_k(s,t)]-\psi(2^{-1}(s+t),t)
\end{align*}
and, consequently, recalling that $\psi(r^1_1,r^1_2)=S(2^{-1}(t-s))g(2^{-1}(s+t),t)$ and taking into account that $|2^{-1}(s+t)-a|^{-\gamma}\le 2^{\gamma}(t-s)^{-\gamma}$, we can estimate
\begin{align}
\|\widetilde M_n(s,t)-\psi(s,t)\|_{X_{\alpha+\varepsilon}} \leq & C_1(\|\delta_{S,2}g\|_{\mathscr{C}^{\mu}_{-\gamma}((a,b]^3_<;X_\beta)}+\|g\|_{\mathscr{C}^{\eta}_{-\gamma}((a,b]^2_<;X_\alpha)})\notag \\
& \quad\times\bigg (
\sum_{k=1}^{n-1}(2^{k(1-\mu)}\!+\!2^{k(\gamma-\eta)})+1\bigg )|t-s|^{[\eta\wedge(\mu+\beta-\alpha)]-\gamma-\varepsilon}\notag \\
\le &C_2(\|\delta_{S,2}g\|_{\mathscr C^{\mu}_{-\gamma}((a,b]^3_<;X_\beta)}+\|g\|_{\mathscr C^{\eta}_{-\gamma}((a,b]^2_<;X_\alpha)})\notag\\
&\qquad\times
|t-s|^{[\eta\wedge (\mu+\beta-\alpha)]-\gamma-\varepsilon}
\label{magazzini}
\end{align}
for some positive constants $C_1$ and $C_2$ which depend on $\varepsilon$, $\mu$, $b-a$, $\gamma$, $\eta$, $\alpha$ and $\beta$,
since the series $\sum_{k=1}^{\infty}(2^{k(1-\mu)}+2^{k(\gamma-\eta)})$ converges.
Further, from \eqref{sfumature} we get
\begin{align*}
\|\widetilde M_n(s,t)-M_n(s,t)\|_{X_{\alpha+\varepsilon}}
\leq 2^{-n\eta}L_{\alpha,\alpha+\varepsilon}\|g\|_{\mathscr C^{\eta}_{-\gamma}((a,b]^2_{<};X_{\alpha})}(s-a)^{-\gamma}|t-s|^{\eta-\varepsilon}
\end{align*}
and we conclude that $\widetilde M_n(s,t)$ converges to $M(s,t)$ in $X_{\alpha+\varepsilon}$, as $n$ tends to $\infty$. Hence, letting $n$ tend to $+\infty$ in \eqref{magazzini}, it follows that the function $k_g=\psi-M$ satisfies the estimate
\begin{align}
\label{stima_int_sing}
\|k_g(s,t)\|_{X_{\alpha+\varepsilon}}
\le & C_2(\|\delta_{S,2}g\|_{\mathscr C^{\mu}_{-\gamma}((a,b]^3_<;X_\beta)}+\|g\|_{\mathscr C^{\eta}_{-\gamma}((a,b]^2_<;X_\alpha)})\notag\\
&\qquad\times|t-s|^{[\eta\wedge (\mu+\beta-\alpha)]-\gamma-\varepsilon}.
\end{align}

Showing that $(\hat\delta_2k_g)(s,t)=0$ is an easy task. Indeed, by the definition of the function $\psi$,
it follows that $\hat\delta_2\psi=\hat\delta_2{\mathbb S}_2g$ in $(a,b]^3_{<}$, which coincides
with $\hat\delta_2M$, as it has been shown at the very beginning of the proof.

Finally, to conclude the proof, we show that function $k_g$ can be extended to $[a,b]^2_{<}$ with a continuous function. First of all, we observe that $k_g$ is continuous in $(a,b]^2_{<}\setminus\{(s,s):s\in [a,b]\}$ since $\psi$ and $M$ are therein continuous. Moreover, using estimate \eqref{stima_int_sing}, we can extend the function $k_g$ by continuity to the points $(s,s)$ with $s\in (a,b]$ by setting $k_g(s,s)=0$. Next, we observe that, for every $a<r<s<t$, it holds that
\begin{align*}
k_g(r,t)-k_g(s,t)
=(\hat\delta_2 k_g)(r,s,t)+S(t-s)k_g(r,s)=S(t-s)k_g(r,s).
\end{align*}
From \eqref{stima_int_sing} we infer that
\begin{align*}
\|k_g(r,t)-k_g(s,t)\|_{X_{\alpha+\varepsilon}}
\leq L_{\alpha+\varepsilon,\alpha+\varepsilon}\|k_g(r,s)\|_{X_{\alpha+\varepsilon}}
\leq C|s-r|^{[\eta\wedge (\mu+\beta-\alpha)]-\gamma-\varepsilon}.
\end{align*}
This implies that the $k_g(s,t)$ converges in $X_{\alpha+\varepsilon}$ as $s$ tends to $a^+$. We denote the previous limit by $k_g(a,t)$. As a byproduct, \eqref{stima_int_sing} holds true for every $(s,t)\in(a,b]^2_<\setminus\{(a,a)\}$ and, using this formula we can extend by continuity $k_g$ at $(a,a)$ setting $k_g(a,a)=0$. It follows that $\hat\delta_2k_g\equiv 0$ in $[a,b]^3_{<}$.

It remains to prove the continuity of $k_g$ in $\{a\}\times [a,b]$. Fix $t_0\in[a,b]$ and $(s,t)\in[a,b]^2_{<}$. We show that $k_g(s,t)$ converges to $k_g(a,t_0)$ in $X_{\alpha+\varepsilon}$ as $(s,t)$ tends to $(a,t_0)$ in $[a,b]^2_<$. First, we consider the case $t_0>a$. If $t>t_0$, then, since $(\hat\delta_2k_g)(a,s,t)=(\hat\delta_2k_g)(a,t_0,t)=0$, it follows that
$k_g(a,t)=k_g(s,t)+S(t-s)k_g(a,s)$ 
and $k_g(a,t)=k_g(t_0,t)+S(t-t_0)k_g(a,t_0)$.
Hence,
\begin{align*}
k_g(s,t)-k_g(a,t_0)
= & k_g(s,t)-k_g(a,t)+k_g(a,t)-k_g(a,t_0) \\
=& -S(t-s)k_g(a,s)+k_g(t_0,t)+S(t-t_0)k_g(a,t_0)-k_g(a,t_0)\\
= & -S(t-s)k_g(a,s)+k_g(t_0,t)+\mathfrak a(t_0,t)k_g(a,t_0)
\end{align*}
On the other hand, if  $t<t_0$, then we can split 
\begin{align*}
k_g(s,t)-k_g(a,t_0)
= & k_g(s,t)-k_g(a,t)+k_g(a,t)-k_g(a,t_0) \\
= & -S(t-s)k_g(a,s)-(k_g(t,t_0)+\mathfrak a(t,t_0)k_g(a,t)),
\end{align*}
where we have used the formulas $k_g(a,t)=k_g(s,t)+S(t-s)k_g(a,s)$ 
and $k_g(a,t_0)=k_g(t,t_0)+k_g(a,t)+\mathfrak a(t,t_0)k_g(a,t)$. Combining the cases $t>t_0$ and $t<t_0$ it follows that
\begin{align}
\|k_g(s,t)-k_g(a,t_0)\|_{X_{\alpha+\varepsilon}}
\leq & L_{\alpha+\varepsilon,\alpha+\varepsilon}\|k_g(a,s)\|_{X_{\alpha+\varepsilon}}+\|k_g(t_0\wedge t,t_0\vee t)\|_{X_{\alpha+\varepsilon}}\notag \\
& +\|\mathfrak a(t_0\wedge t,t_0\vee t)k_g(a,t_0\wedge t)\|_{X_{\alpha+\varepsilon}}.
\label{prova}
\end{align}
From \eqref{stime_smgr}$(b)$ and estimate \eqref{stima_int_sing}, we get
\begin{align}
&\|\mathfrak a(t_0\wedge t,t_0\vee t)k_g(a,t_0\wedge t)\|_{X_{\alpha+\varepsilon}}\notag\\
\le &C_{\alpha+\varepsilon,\alpha+\varepsilon'}\|k_g(a,t_0\wedge t)\|_{X_{\alpha+\varepsilon'}}|t-t_0|^{\varepsilon'-\varepsilon}\notag\\
\le &C_2C_{\alpha+\varepsilon,\alpha+\varepsilon'}(\|\delta_{S,2}g\|_{\mathscr C^{\mu}_{-\gamma}((a,b]^3_<;X_\beta)}+\|g\|_{\mathscr C^{\eta}_{-\gamma}((a,b]^2_<;X_\alpha)})\notag\\
&\qquad\times|b-a|^{[\eta\wedge (\mu+\beta-\alpha)]-\gamma-\varepsilon}|t-t_0|^{\varepsilon'-\varepsilon}
\label{itinere}
\end{align}
for every $\varepsilon'\in (\varepsilon,\varepsilon_0)$.
From \eqref{stima_int_sing}, \eqref{prova} and \eqref{itinere}  we easily conclude that $k_g(s,t)$ tends to $k_g(a,t_0)$ as $(s,t)$ tends to $(a,t_0)$.

Finally, if $t_0=a$, then, since
$\|k_g(s,t)-k_g(a,a)\|_{X_{\alpha+\varepsilon}}
=\|k_g(s,t)\|_{X_{\alpha+\varepsilon}}$, using 
\eqref{stima_int_sing} we conclude that $k_g(s,t)$ converges to $k(a,t_0)$ as $(s,t)$ tends to $(a,t_0)$. \end{proof}

\begin{remark}
\label{lin-int}
{\rm 
\begin{enumerate}[\rm (i)]
\item
From the proof of Theorem \ref{thm:sew_map_sing_int}, it follows that $k_{g_1+g_2}=k_{g_1}+k_{g_2}$ for every pair of functions $g_1,g_2\in \mathscr{C}^{\eta}_{-\gamma}((a,b]^2_{<};X_{\alpha})$ such that
$\delta_{S,2}g_1$ and $\delta_{S,2}g_2$ belong to $\mathscr{C}^{\mu}_{-\gamma}((a,b]^3_{<};X_{\beta})$, where the parameters $\alpha$, $\beta$, $\gamma$, $\eta$ and $\mu$ are as in the statement of the quoted theorem.
\item
Still from the proof of Theorem \ref{thm:sew_map_sing_int} it follows that condition \eqref{stime_smgr}(a) is used just to prove that $k_g$ regularizes in space. Hence, without such a condition and assuming that $\beta=\alpha$, the assertion of Theorem \ref{thm:sew_map_sing_int} still holds true with $\varepsilon=0$, i.e., the function $k_g$ exists, belongs  to $\mathscr C^{\eta-\gamma}([a,b]^2_<;X_\alpha)$ and enjoys estimate \eqref{stima_int_sing} with $\varepsilon=0$ and the constant $C$, therein appearing, depends on $b-a$, $\eta$, $\gamma$ and $\alpha$. 
\end{enumerate}}
\end{remark}

Based on Theorem \ref{thm:sew_map_sing_int} and Remark \ref{rmk:int_sing_confr_caso_classico}, 
we can now give the following definition, which generalizes Definition \ref{def-2.8}.

\begin{definition}
Let $g\in \mathscr{C}^{\eta}_{-\gamma}((a,b]_<^2;X_\alpha)$ be such that $\delta_{S,2}g\in \mathscr{C}^{\mu}_{-\gamma}((a,b]^3_<;X_\beta)$, where 
$\eta\in (0,1)$, $\mu>1$, $0\leq \beta\leq \alpha\leq 2$, $\alpha-\beta<1$ and $\gamma\in[0,\eta\wedge(\mu+\beta-\alpha))$. Then, the function ${\mathscr I}_g=k_g(a,\cdot)$
is called {\textbf{convolution integral of $\bm{g}$}}.
\end{definition}

\begin{example}
\label{example:sing_int_1}
{Fix $\eta\in(0,1)$, $\alpha,\beta\in [0,2)$ such that $0\le \alpha-\beta<1$, $\rho\in(1-\eta,1)$, $\gamma\in[0,\eta\wedge (\rho+\eta+\beta-\alpha))$, $x\in C^\eta([a,b])$ and $\varphi\in \mathscr{C}_{-\gamma}((a,b];X_\alpha)$ such that $\hat\delta_1\varphi\ \mathscr{C}^{\rho}_{-\gamma}((a,b]^2_<;X_{\beta})$. Let $g(s,t)=(x(t)-x(s))\varphi(s)$ for every $(s,t)\in[a,b]^2_<$, as in Example $\ref{ex:new_sew_map}$.
It is easy to check that $g$ belongs to $\mathscr{C}^{\eta}_{-\gamma}((a,b]^2_<;X_\alpha)$. Moreover, Example $\ref{ex:new_sew_map}$ shows that $(\delta_{S,2}g)(r,s,t)=-(x(t)-x(s))(\hat\delta_1\varphi)(r,s)$ for every $(r,s,t)\in (a,b]^3_{<}$. Hence, the function $\delta_{S,2}g$ belongs to 
$\mathscr{C}^{\mu}_{-\gamma}((a,b]^3_<;X_{\beta})$ where $\mu=\eta+\rho>1$, and
\begin{align*}
(r-a)^\gamma\|(\delta_{S,2}g)(r,s,t)\|_{X_{\beta}}\leq \|x\|_{C^{\eta}([a,b])}\|\hat\delta_1\varphi\|_{\mathscr{C}^{\rho}_{-\gamma}((a,b]^2_<;X_{\beta})}|t-r|^{\mu},
\end{align*}
for every $(r,s,t)\in(a,b]^3_<$.}
\end{example}

In view of Theorem \ref{thm:sew_map_sing_int} and Example 
\ref{example:sing_int_1}, we can give the following definition.

\begin{definition}
\label{def-pioggia}
Fix $\alpha,\beta\in [0,2)$, with $0\le \alpha-\beta<1$, $\eta\in (0,1)$, $\rho\in(1-\eta,1)$, $\gamma\in [0,\eta\wedge(\rho+\eta+\beta-\alpha))$ and $x\in C^\eta([a,b])$. Further, let $\varphi\in \mathscr{C}_{-\gamma}((a,b];X_{\alpha})$ be such that $\hat\delta_1\varphi\in \mathscr{C}^{\rho}_{-\gamma}((a,b]^2_{<};X_{\beta})$. Then, we define the convolution integral of the semigroup $(S(t))_{t\ge 0}$ with the function $\varphi$, by setting 
\begin{align*}
\int_s^tS(t-r)\varphi(r)dx(r)=:k_g(s,t),
\end{align*}
for every $(s,t)\in [a,b]^2_{<}$,
where $g(s,t)=(x(t)-x(s))\varphi(s)$ for every $(s,t)\in [a,b]^2_{<}$.
\end{definition}

In what follows, for every $(s,t)\in [a,b]^2_{<}$ we will also use  the shorter notation
${\mathscr I}_{S,\varphi}(s,t):=k_g(s,t)$ to denote the convolution integral in  Definition \ref{def-pioggia}. The notation underline the dependence on the semigroup $(S(t))_{t\ge 0}$ and on the function $\varphi$.

\begin{remark}
\label{lemma:def_int_conv}
{\rm Let $x$ and $\varphi$ be as in Definition \ref{def-pioggia}. Then, from Theorem \ref{thm:sew_map_sing_int} it follows that ${\mathscr I}_{S,\varphi}$ belongs to $\displaystyle\bigcap_{\varepsilon\in [0,\varepsilon_0)}\mathscr{C}^{\eta\wedge (\eta+\rho+\beta-\alpha)-\gamma-\varepsilon}([a,b]^2_<;X_{\alpha+\varepsilon})$, where
$\varepsilon_0=(1+\beta-\alpha)\wedge (\eta\wedge (\eta+\rho+\beta-\alpha)-\gamma)$,
and, for every $\varepsilon\in [0,\varepsilon_0)$, there exists a positive constant $C$, which depends on $\varepsilon$, $b-a$, $\eta$, $\gamma$, $\alpha$, $\beta$ and $\rho$, such that
\begin{align}
&\sup_{(t,s)\in [a,b]^2_{<}}(t-s)^{-\eta\wedge (\eta+\rho+\beta-\alpha)+\gamma+\varepsilon}\bigg\|\int_s^tS(t-r)\varphi(r)dx(r)\bigg\|_{X_{\alpha+\varepsilon}}\notag\\
\le &C\|x\|_{C^{\eta}([a,b])}
(\|\varphi\|_{\mathscr{C}_{-\gamma}((a,b];X_{\alpha})}+\|\hat\delta_1\varphi\|_{\mathscr{C}^{\rho}_{-\gamma}((a,b]^2_{<};X_{\beta})})
\label{stima-utilee}
\end{align}
for every $s,t\in [a,b]^2_{<}$.}
\end{remark}

\begin{remark}
\label{rmk:int_sing_confr_caso_classico}
{\rm We stress that, if $x\in C^1([a,b])$, $\varphi\in \mathscr{C}_{-\gamma}((a,b];X_\alpha)$ and $\hat\delta_{1}\varphi\in \mathscr{C}^{\alpha}_{-\gamma}((a,b]^2_<;X)$ with $\alpha\in(0,1)$ and $\gamma$ satisfying the assumptions of the quoted theorem, then the function ${\mathscr I}_{S,\varphi}$ is the classical convolution of the semigroup with the function $\varphi$, i.e.,
\begin{align}
\label{int_caso_classico}
{\mathscr I}_{S,\varphi}(s,t)=\int_s^tS(t-\xi)\varphi(\xi)x'(\xi)d\xi, \qquad\;\, (s,t)\in[a,b]^2_<.
\end{align}

Note that
\begin{align*}
\int_s^tS(t-\xi)\varphi(\xi)x'(\xi)d\xi= & \int_s^t S(t-\xi)\big[ S(\xi-s)\varphi(s)+(\hat\delta_1\varphi)(s,\xi)\big]x'(\xi)d\xi \\
=&S(t-s)g(s,t)+\int_s^tS(t-\xi)(\hat\delta_1\varphi)(s,\xi)x'(\xi)d\xi \\
=&\!: S(t-s)g(s,t)+N(s,t)
\end{align*}
for every $(s,t)\in [a,b]^2_{<}$. 

Thus, we need to prove that $M+N=0$ where $M$ is defined in Example \ref{ex:new_sew_map}, see also Proposition \ref{prop:new_sew_map}.

Taking the definition of the operators $\hat\delta_1$, $\mathbb S_3$ and formula \eqref{pixel} into account, we can write
\begin{align*}
(\hat\delta_2N)(r,s,t)=&\int_r^tS(t-\xi)(\hat\delta_1\varphi)(r,\xi)x'(\xi)d\xi
-\int_s^tS(t-\xi)(\hat\delta_1\varphi)(s,\xi)x'(\xi)d\xi \\
& -S(t-s)\int_r^sS(s-\xi)(\hat\delta_1\varphi)(r,\xi)x'(\xi)d\xi\\
= & \int_s^tS(t-\xi)(\hat\delta_1\varphi)(r,\xi)x'(\xi)d\xi-\int_s^tS(t-\xi)(\hat\delta_1\varphi)(s,\xi)x'(\xi)d\xi\\
=&\int_s^t(S(t-s)\varphi(s)-S(t-r)\varphi(r))x'(\xi)d\xi\\
=&S(t-s)(\varphi(s)-S(s-r)\varphi(r))(x(t)-x(s))\\
=&S(t-s)(\hat\delta_1\varphi)(r,s)(x(t)-x(s))\\
=&-(\mathbb S_3\delta_{S,2}g)(r,s,t)
\end{align*}
for every $(r,s,t)\in [a,b]^3_{<}$. Thus, the function  $A=M+N$ satisfies the condition $\hat\delta_2A\equiv 0$ in $[a,b]^3_{<}$.

To conclude that $A$ identically vanishes in $(a,b]^2_{<}$, it suffices to show that it belongs to
$\mathscr{C}^{\mu}([a+\theta,b]^2_{<};X)$ for every $\theta\in (0,b-a)$ and then apply \cite[Proposition 3.4]{GT10}. By Lemma \ref{lem:function_M}, $M$ belongs to such a space. On the other hand,
\begin{align*}
\|N(s,t)\|_X
\leq & \theta^{-\gamma} L_{0,0,b-a}\|x'\|_{C([a,b])}\|\hat\delta_1\varphi\|_{\mathscr{C}^{\alpha}_{-\gamma}((a,b]^2_<;X)}\int_s^t(\xi-s)^{\alpha}dr \\
= & \theta^{-\gamma} L_{0,0,b-a}(\alpha+1)^{-1}\|x'\|_{C([a,b])}\|\hat\delta_1\varphi\|_{\mathscr{C}^{\alpha}_{-\gamma}((a,b]^2_<;X)}(t-s)^{1+\alpha}
\end{align*}
for every $(s,t)\in[a+\theta,b]^2_<$, with $\theta\in(0,b-a)$, and every $\varepsilon\in (0,1)$, and, consequently, $N$ belongs to $\mathscr{C}^{1+\alpha}([a+\theta,b]^2_{<};X)\hookrightarrow \mathscr{C}^{\mu}([a+\theta,b]^2_{<};X)$ for every $\theta\in (0,b-a)$, since $1+\alpha>\eta+\alpha=\mu$.}
\end{remark}

\begin{remark}
{\rm The results in Theorem \ref{thm:sew_map_sing_int} are optimal as far as both the time and the spatial regularity are concerned.

Indeed, if we refer again to Remark \ref{rmk:int_sing_confr_caso_classico}, then 
the classical convolution integral in \eqref{int_caso_classico} belongs to $\mathscr{C}^{1-\varepsilon-\gamma}([a,b]^2_<;X_{\alpha+\varepsilon})$ for every $\varepsilon\in[0,1-\gamma)$. On the other hand, the quoted theorem shows that ${\mathscr I}_{S,\varphi}$ belongs to $\mathscr{C}^{\eta-\varepsilon-\gamma}([a,b]^2_<;X_{\alpha+\varepsilon})$ for every $\eta\in (0,1)$ and $\varepsilon$ as above, and the constant 
$C$ appearing in \eqref{stima_sew_map_sing} does not blow up as $\eta$ tends to $1$ from below. Therefore, 
in this situation where $x\in C^1([a,b])$, estimate
\eqref{stima_sew_map_sing} shows that ${\mathscr I}_{S,\varphi}$ belongs to $\mathscr{C}^{1-\gamma-\varepsilon}(([a,b]^2_{<};X_{\alpha+\varepsilon})$
for every $\varepsilon\in[0,1-\gamma)$, so that
the time regularity in Theorem \ref{thm:sew_map_sing_int} is optimal.

To prove also the spatial optimality of the result in Theorem \ref{thm:sew_map_sing_int}, we consider the case when $(S(t))_{t\ge 0}$ is an analytic semigroup in the Banach space $X$ and $X_{\beta}=D_A(\beta,\infty)$ for every $\beta\in (0,1)$. If we choose $[a,b]=[0,1]$, $x(t)=t$ for every $t\in [0,1]$ and $\varphi(t)=S(t)y$ for some $y\in X_{\alpha-\gamma}$, where $\gamma$ is fixed in $(0,\alpha)$ for some $\alpha\in (0,1)$, and $y$ does not belong to any space
$X_{\beta}$ with $\beta>\alpha-\gamma$, then we can easily check that the classical convolution $k_g$ is given by $k_g(s,t)=(t-s)S(t)y$ for every $0\le s\le t\le 1$ and it does not belong to the space $\mathscr{C}^{\delta}([0,1]^2_{<};X_{\alpha+\varepsilon})$ if $\delta>1-\varepsilon-\gamma$ for every $\varepsilon\in (0,1-\gamma)$.
Indeed, suppose that this function belongs to $\mathscr{C}^{\delta}([0,1]^2_{<};X_{\alpha+\varepsilon})$ for some $\delta>1-\varepsilon-\gamma$. Then, in particular,
\begin{eqnarray*}
\|S(t)y\|_{X_{\alpha+\varepsilon}}=\|k_g(0,t)\|_{X_{\alpha+\varepsilon}}\le \|k_g\|_{C_{\delta}([0,1]^2_{<};X_{\alpha+\varepsilon})}t^{\delta-1},\qquad\;\,t\in (0,1].
\end{eqnarray*}
We claim that this estimate implies that
$y$ actually belongs to the space $X_{\alpha+\varepsilon+\delta-1}$. Since
$\alpha+\varepsilon+\delta-1>\alpha-\gamma$ we are led to a contradiction.

To prove the claim, we recall that $z$ belongs to $X_{\beta}=D_A(\beta,\infty)$ for some $\beta\in (0,1)$ if and only if $\displaystyle\sup_{t\in (0,1]}t^{1-\beta}\|AS(t)z\|_X<\infty$. Using the semigroup law, we can estimate
\begin{align*}
t^{2-\alpha-\varepsilon-\delta}\|AS(t)y\|_X
=&t^{2-\alpha-\varepsilon-\delta}\|AS(t/2)S(t/2)y\|_X\\
\le &t^{2-\alpha-\varepsilon-\delta}\|AS(t/2)\|_{L(X_{\alpha+\varepsilon},X)}\|S(t/2)y\|_{X_{\alpha+\varepsilon}}\\
\le & Ct^{2-\alpha-\varepsilon-\delta}t^{-1+\alpha+\varepsilon}t^{\delta-1}=C
\end{align*}
for every $t\in (0,1]$ and the claim follows.}
\end{remark}

\begin{remark}
\label{rmk-abbruzzetti}
{\rm If $\gamma=0$, i.e., in the non singular case, Example \ref{example:sing_int_1} shows that Theorem \ref{thm:sew_map_sing_int} agrees with the results in \cite[Lemma 2.1]{ALT}.}
\end{remark}

\begin{remark}
\label{rmk:prop_int_young}
{\rm 
\begin{enumerate}[\rm(i)]
\item 
From Remark \ref{lin-int}(i) it follows that, if $x_1,x_2\in C^\eta([a,b])$, for some $\eta\in (0,1)$, and $\varphi_1,\varphi_2\in \mathscr{C}_{-\gamma}((a,b];X_{\alpha})$ are such that
$\hat\delta_1\varphi_1, \hat\delta_1\varphi_2\in \mathscr{C}^{\rho}_{-\gamma}((a,b]^2_{<};X_\beta)$ for some $\alpha,\beta\in[0,2)$ with $0\leq \alpha-\beta<1$, $\rho\in (1-\eta,1)$ and $\gamma\in[0,\eta\wedge (\eta+\rho+\beta-\alpha))$, then
\begin{align*}
&\int_s^tS(t-r)\varphi_1(r)d(x_1+x_2)(r)\\
=&\int_s^tS(t-r)\varphi_1(r)dx_1(r)+\int_s^tS(t-r)\varphi_1(r)dx_2(r)
\end{align*}
and
\begin{align*}
&\int_s^tS(t-r)(\varphi_1(r)\!+\!\varphi_2(r))dx_1(r)\\
=&\int_s^tS(t-r)\varphi_1(r)dx_1(r)\!+\!\int_s^tS(t-r)\varphi_2(r)dx_1(r)
\end{align*}
for every $(s,t)\in [a,b]^2_{<}$.
\item
For every every $x\in C^{\eta}([a,b])$, with $\eta\in (0,1)$, and every $\varphi\in \mathscr{C}_{-\gamma}((a,b];X_{\alpha})$ such that $\hat\delta_1\varphi\in \mathscr{C}^{\rho}_{-\gamma}((a,b]^2;X_\beta)$ for some $\alpha,\beta\in[0,2)$ with $0\leq \alpha-\beta<1$, $\rho\in (1-\eta,1)$ and $\gamma\in[0,\eta\wedge (\eta+\rho+\beta-\alpha))$,it holds that
\begin{align*}
\int_s^tS(t-r)\varphi(r)dx(r)=&S(t-\tau)\int_s^\tau S(\tau-r)\varphi(r)dx(r)\\
&+\int_\tau^tS(t-r)\varphi(r)dx(r)
\end{align*}
for every $(s,t)\in [a,b]^2_{<}$ and $\tau\in[s,t]$.
This property is a straightforward rewriting of the property $\hat\delta_2 k_g\equiv 0$ follows easily from observing that $\hat\delta_1{\mathscr I}_{S\varphi}$ identically vanishes in $[a,b]^3_{<}$.
\end{enumerate}
}
\end{remark}

\subsection{The case $\bm{S(t)=Id}$}

In this section, taking advantage of the results of the previous subsection, we define the integral 
$\displaystyle\int_s^t\varphi(r)dx(r)$, when $\varphi$ has a singularity at the left-endpoint of the interval where it is defined.

The main result is the following theorem.

\begin{theorem}
\label{coro:sing_int}
Fix $\alpha\in [0,2)$. Assume that $x\in C^{\eta}([a,b])$, for some $\eta\in (0,1)$, $\varphi\in \mathscr{C}_{-\gamma}((a,b];X_{\alpha})$ and $\delta_1\varphi\in \mathscr{C}^{\rho}_{-\gamma}((a,b]^2_{<};X_{\alpha})$, for some $\gamma\in (0,\eta)$ and some $\rho\in (1-\eta,1)$. Then, the Young integral
\begin{align*}
\int_s^t\varphi(r)dx(r)    
\end{align*}
is well defined for every $(s,t)\in [a,b]^2_{<}$. Moreover, there exists a positive constant $C$, depending on $\alpha$, $\gamma$, $\eta$, $\rho$ and $b-a$, such that
\begin{align*}
\bigg\|\int_s^t\varphi(r)dx(r)    
\bigg\|_{X_{\alpha}}\!\!\!\!
\le &
C\|x\|_{C^{\eta}([a,b])}(\|\varphi\|_{\mathscr{C}_{-\gamma}((a,b];X_{\alpha})}+\|\delta_1\varphi\|_{\mathscr{C}^{\rho}_{-\gamma}((a,b]^2_{<};X_{\alpha})}) |t-s|^{\eta-\gamma}.
\end{align*}
\end{theorem}

\begin{proof}
It suffices to apply Theorem \ref{thm:sew_map_sing_int}, with $S(t)=I$ for every $t>0$, observing that $\hat\delta_j=\delta_j$ for $j=1,2$ and $\delta_{S,2}=\delta_2$, and taking Remark \ref{lin-int}(ii) into account. Note that condition \eqref{stime_smgr}(b) is trivially satisfied since $\mathfrak a(s,t)=0$ for every $(s,t)\in[a,b]^2_<$.
\end{proof}

\begin{example}
Let $x$ belong to $C^\eta([0,T])$ for some $T>0$ and $\eta\in (0,1)$. For every $\alpha\in (0,1)$ the function $f:(0,T]\to\R$, defined by $f(t)=t^{-\alpha}$ for every $t\in (0,T]$, belongs to $ \mathscr{C}^{\alpha}_{-2\alpha}((0,T];\R)$. Indeed, it is easy to check that
\begin{align*}
\frac{|t^{-\alpha}-s^{-\alpha}|}{(t-s)^\alpha}s^{2\alpha}
=\frac{t^{\alpha}-s^{\alpha}}{(t-s)^\alpha}s^{\alpha}t^{-\alpha}
\leq s^{\alpha}t^{-\alpha}\leq 1, \qquad\;\, 0<s<t\leq T.
\end{align*}
Therefore, if we take $\alpha$ and $\eta$ such that $\alpha<\frac{1}{2}\eta$ and $\alpha+\eta>1$, then the assumption of Theorem $\ref{coro:sing_int}$ are fulfilled and the integral
\begin{align*}
\int_s^tr^{-\alpha}dx(r)
\end{align*}
is well defined for every $s,t\in [0,T]$, with $s<t$.
\end{example}

\section{Mild solutions to Young equations}
\label{sect-3}
In this section we study the existence and uniqueness of the mild solution to the following nonlinear Young equation
\begin{align}
\label{omo_equation}
\left\{
\begin{array}{ll}
dy(t)=Ay(t)dt+\sigma(y(t))dx(t), &t\in(0,T],\\[1mm]
y(0)=\psi,
\end{array}
\right.
\end{align}
where by mild solution we mean a function $y:[0,T]\to X$ such that $\mathscr I_{S,\sigma\circ y}$ is well-defined in $[0,T]^2_<$ and
\begin{equation}
y(t)=S(t)\psi+({\mathscr I}_{S,\sigma\circ y})(0,t),\qquad\;\,t\in [0,T].
\label{reg_sol-aaa}
\end{equation}
Further, we investigate the spatial smoothness of the mild solution. 

\begin{remark}
{\rm Let us observe that $\mathscr{I}_{S,\sigma\circ y}(0,t)$ converges to $0$ in $X$ as $t$ tends to $0$, due to Remark \ref{lemma:def_int_conv}.
The continuity at $0$ of the term $S(\cdot)\psi$ requires a more detailed discussion.
If the semigroup $(S(t))_{t\ge 0}$ is strongly continuous 
then, for every $\psi\in X$ (which is allowed by Theorem \ref{thm:ex_mild_sol_1}), $S(t)\psi$ converges in $X$ to $\psi$ as $t$ tends to $0$ and the initial condition can be classically interpreted. The point is that, under our assumptions, the semigroup could be not strongly continuous. In this situation, if $\psi\in X_{\theta}$ for some $\theta>0$, then, due to condition
\eqref{stime_smgr}(b), $S(t)\psi$ converges to $\psi$ in $X$ as $t$ tends to $0$ and again  the initial condition can be classically interpreted. On the contrary, if $\psi$  only belongs to $X$, then, in general, $S(t)\psi$ does not admit limit as $t$ tends to $0$. The initial condition is satisfied by the mild solution in this sense: for every $t_0>0$, $S(t_0)y(t)$ converges in $X$ to $S(t_0)\psi$. Indeed, by Remark \ref{rmk-2.2}, the function $t\mapsto S(t)x$ is continuous in $(0,+\infty)$ for every $x\in X$.

We finally observe that, if $(S(t))_{t\ge 0}$ is an analytic semigroup and $\psi\in\overline{D(A)}$, then $y(t)$ converges to $\psi$ in $X$ as $t$ tends to $0$, If $\overline{D(A)}$ is a proper subspace of $X$ (i.e., the semigroup is not strongly continuous) and $\psi$ does not belong to it, then the initial condition can also be interpreted as follows: for every $\lambda\in\rho(A)$, $R(\lambda,A)y(t)$ converges in $X$ to $R(\lambda,A)\psi$ as $t$ tends to $0$.}
\end{remark}

We split this section into three parts. In the former we prove the results when, among other properties, the nonlinear term $\sigma$ is globally Lipschitz continuous in $X_\alpha$, for some $\alpha\in(0,1)$. In the other two subsections we prove the same results by assuming that $\sigma$ is only locally Lipschitz continuous in $X_\alpha$ and that $\sigma'$ is locally Lipschitz continuous from $X_\alpha$ into $X$, respectively, for some $\alpha\in(0,1)$. In these cases we need to strengthen the hypotheses on $\eta$ in order to balance the lack of regularity of $\sigma$. To simplify the computations we consider $T=1$, since the general case can be obtained with analogous arguments. We stress that, for arbitrary $T>0$, the constants which appear in the estimates also depend on $T$.

\subsection{The case when $\sigma$ is globally Lipschitz continuous in $X_\alpha$}
\label{subsect-3.1}
We stress that, even if the following set of assumptions on $\sigma$ might seem a bit artificial (since we assume $\sigma$ to be Lipschitz both in $X$ and $X_{\alpha}$ while its derivative $\sigma'$ is only assumed to be locally Lipschitz continuous).  We have two reasons to also consider this case. The former is that the proofs, although maintaining the main difficulties,  are,  in this setting,  easier and should help the reader to better understand the ideas behind the computations, the latter is that the proof of the main results under weaker hypotheses on $\sigma$ can be deduced from the computations developed in this subsection, with some slight modifications. Hence, we can see this part as an intermediate step in order to prove more general statements. 
\begin{hypotheses}
\label{ip_nonlineare}
\begin{enumerate}[\rm (i)]
\item
Hypotheses $\ref{hyp-main}$ are satisfied for every $\lambda,\zeta$ such that $0\le\zeta\le\lambda<2$.
\item
The function $x$ belongs to $C^{\eta}([0,1])$ for some $\eta\in (1/2,1)$.
\item 
The function $\sigma:X\to X$ is G\^ateaux differentiable with bounded and locally Lipschitz continuous G\^ateaux derivative $\sigma'$. We denote by ${\rm Lip}_\sigma$ the Lipschitz constant of $\sigma$ on $X$ and, for every $R>0$, by ${\rm Lip}_{\sigma'}^R$ the Lipschitz constant of $\sigma'$ in $\{z\in X:\|z\|_X\leq R\}$.
\item
There exists $\alpha\in(0,1)$ such that $\eta+\alpha>1$, the restriction of $\sigma$ to $X_{\alpha}$ maps this space into itself, and $\sigma$ is Lipschitz continuous as a map from $X_\alpha$ into itself. We denote by ${\rm Lip}_\sigma^\alpha$ the Lipschitz constant of $\sigma$ as a map from $X_\alpha$ into itself.
\end{enumerate}
\end{hypotheses}

Let us introduce the following space: for every $\alpha,\gamma>0$, we say that $f\in Y^{\alpha}_{-\gamma}(0,1)$ if
$f\in \mathscr C_{-\gamma}((0,1];X_\alpha)\cap C_b((0,1];X)$ and $\hat\delta_1 f\in \mathscr C^{\alpha}([0,1]^2_<;X_{\alpha})$. The space $Y^{\alpha}_{-\gamma}(0,1)$ is a Banach space if endowed with the norm
\begin{align*}
\|f\|_{Y^{\alpha}_{-\gamma}(0,1)}:=\|f\|_{\mathscr C_{-\gamma}((0,1];X_\alpha)}+\|f\|_{C_b((0,1];X)}+\|\hat\delta_1f\|_{\mathscr C^{\alpha}([0,1]^2_<;X_{\alpha})}
\end{align*}
for every $f\in Y^{\alpha}_{-\gamma}(0,1)$.

An analogous definition is given for $Y^{\alpha}_{-\gamma}(0,T)$ with $T\in (0,1)$

\begin{remark}
{\rm One may ask why, if $f \in Y^{\alpha}_{-\gamma}(0,1)$, then the function $t\mapsto \|f(t)\|_{X_\alpha}$ has a singularity of order $\gamma$ at $t=0$, while the function $(s,t)\mapsto \|\hat\delta_1 f(s,t)\|_{X_\alpha}$ has not a singularity at $s=0$.  The reason is the following: if  $\psi\in X_\theta$ and $\theta<\alpha$, then the map $t\mapsto \|S(t)\psi\|_{X_\alpha}$ has a singularity at $t=0$ of order $\alpha-\theta$. On the contrary,  $\hat\delta_1 (S(\cdot)\psi)=0$ for every $\psi\in X$.
Thus, if $y$ is a mild solution to \eqref{omo_equation}, then $\|y(t)\|_{X_{\alpha}}$ has a singularity at $t=0$, while $(\hat\delta_1 y)(s,t)={\mathscr I}_{S,\sigma\circ y}(s,t)$ for every $(s,t)\in [0,1]^2_<$, and so it has no singularity with respect to the ${X_\alpha}$-norm.}
\end{remark}

To begin with, we prove that the convolution integral $\mathscr I_{S(\sigma\circ y)}$ is well-defined for functions $y\in Y^{\alpha}_{\theta-\alpha}(0,1)$. 
\begin{lemma}
\label{lemma:int_conv_funz_Y_alpha_gamma}
Let Hypotheses $\ref{ip_nonlineare}$ be satisfied and fix $y\in Y^{\alpha}_{-\gamma}(0,1)$ with $\eta+\alpha>1$ and $\eta>\gamma$. Then:
\begin{enumerate}[\rm(i)]
\item 
$\mathscr I_{S,\sigma\circ y}$ is well-defined;
\item 
$\mathscr I_{S,\sigma\circ y}$ belongs to $\mathscr C^{\eta-\gamma-\varepsilon}([0,1]^2_<;X_{\alpha+\varepsilon})$ 
and
\begin{align*}
\|\mathscr I_{S,\sigma\circ y}\|_{\mathscr C^{\eta-\gamma-\varepsilon}([0,1]^2_{<};X_{\alpha+\varepsilon})}
\leq & C \|x\|_{C^{\eta}([0,1])}
(1+\|y\|_{Y^\alpha_{-\gamma}(0,1)})
\end{align*}
for every $\varepsilon\in[0,(1-\alpha)\wedge (\eta-\gamma)]$ and some positive constant $C$, which depends on   $\varepsilon$, $\eta$, $\gamma$, the constant $C_{\alpha,0,1}=C_{\alpha,0}$ in \eqref{stime_smgr}$(b)$, ${\rm Lip}_{\sigma}$, ${\rm Lip}_{\sigma}^{\alpha}$ and the norm of $\sigma(0)$ in $X_{\alpha}$.     
\end{enumerate}
\end{lemma}

\begin{proof}
(i). Let us prove that the function $\sigma\circ y$ satisfies the assumptions of Remark \ref{lemma:def_int_conv} with $a=0$, $b=1$, $\gamma=\alpha-\theta$, $\rho=\alpha$ and $\beta=0$, i.e., $\sigma\circ y\in \mathscr C_{-\gamma}((0,1];X_\alpha)$ and $\hat\delta_1(\sigma\circ y)\in \mathscr C^{\alpha}_{-\gamma}((0,1]^2_<;X)$. 
The continuity of $\sigma\circ y$ and $\hat\delta_1(\sigma\circ y)$ is a consequence of the regularity of $\sigma$ and $y$. Further, from Hypothesis \ref{ip_nonlineare}(iii) it follows that 
\begin{align}
\|\sigma(z)\|_{X_{\alpha}}\le L^{\alpha}_{\sigma}(1+\|z\|_{X_{\alpha}}),\qquad\;\,z\in X_{\alpha},
\label{ip_sigma}
\end{align}
where $L^{\alpha}_{\sigma}=\max\{{\rm Lip}_{\sigma}^{\alpha},\|\sigma(0)\|_{X_{\alpha}}\}$.
Therefore, we infer that
\begin{align}
s^\gamma\|\sigma(y(s))\|_{X_\alpha}
\leq & L_{\sigma}^{\alpha}s^{\gamma}(1+\|y(s)\|_{X_\alpha})
\leq  L_{\sigma}^{\alpha}(1+\|y\|_{\mathscr C_{-\gamma}((0,1];X_\alpha)}).
\label{ip_sigma-1}
\end{align}

Let us prove the condition on $\hat\delta_{1}(\sigma\circ y)$. For this purpose, we observe that
\begin{align*}
\|(\hat\delta_1(\sigma\circ y))(s,t)\|_{X}
\leq & \|\sigma(y(t))-\sigma(y(s))\|_X+\|\mathfrak a(s,t)\sigma(y(s))\|_X  \\
\leq & {\rm Lip}_\sigma\|y(t)-y(s)\|_X+C_{\alpha,0}|t-s|^\alpha\|\sigma(y(s))\|_{X_\alpha} \\
\leq & {\rm Lip}_\sigma\|(\hat\delta_1y)(s,t)\|_{X}+{\rm Lip}_\sigma\|\mathfrak a(s,t)y(s)\|_{X} \\
& +C_{\alpha,0}L^\alpha_\sigma|t-s|^\alpha(1+\|y(s)\|_{X_\alpha}) \\
\leq & {\rm Lip}_\sigma\|(\hat\delta_1y)(s,t)\|_{X}+C_{\alpha,0}|t-s|^\alpha(L^\alpha_\sigma+({\rm Lip}_\sigma+L^\alpha_\sigma)\|y(s)\|_{X_\alpha})
\end{align*}
for every $(s,t)\in (0,1]^2_{<}$.
By recalling the definition of $Y^{\alpha}_{-\gamma}(0,1)$, we conclude that
\begin{align*}
s^\gamma\frac{\|(\hat\delta_{1}(\sigma\circ y))(s,t)\|_{X}}{|t-s|^\alpha}
\leq & C_{\alpha,0}L^\alpha_\sigma +[(C_{\alpha,0}\vee 1){\rm Lip}_\sigma+C_{\alpha,0}L^\alpha_\sigma]\|y\|_{Y^{\alpha}_{-\gamma}(0,1)}.
\end{align*}
Hence, the assumptions of Remark \ref{lemma:def_int_conv} are fulfilled by $\sigma\circ y$ and so the convolution integral $\mathscr I_{S,\sigma\circ y}$ is well-defined.

\medskip
(ii). From (i) and Remark \ref{lemma:def_int_conv},
it follows that ${\mathscr I}_{S,\sigma\circ y}\in \mathscr C^{\eta-\gamma-\varepsilon}([0,1]^2_<;X_{\alpha+\varepsilon})$ for every $\varepsilon\in[0,(1-\alpha)\wedge (\eta-\gamma))$ and estimate \eqref{stima-utilee} holds true, with $\sigma\circ y$ in place of $\varphi$. 
\end{proof}

\begin{theorem}
\label{thm:ex_mild_sol_1}
Let Hypotheses $\ref{ip_nonlineare}$ be satisfied. Then, for every $\psi\in X_\theta$  with $\alpha\geq \theta$ and $\eta>2\alpha-\theta$, there exists a unique mild solution $y$ to \eqref{omo_equation} which belongs to $Y^{\alpha}_{\theta-\alpha}(0,1)$. 
\end{theorem}

\begin{remark}
{\rm If $\theta=\alpha$ then Theorem $\ref{thm:ex_mild_sol_1}$ coincides with \cite[Theorem 3.2]{ALT}. 
On the other side, we can also choose an initial datum $\psi$ barely belonging to $X$. The conditions  $\eta>2\alpha$ and $\alpha+\eta>1$ imply
$\eta >2/3$. In other words, we can drop any regularity requirement on the initial datum as long as we choose a slightly  more regular noise $x$ (but still not differentiable).}
\end{remark}

\begin{proof}[Proof of Theorem $\ref{thm:ex_mild_sol_1}$]
To begin with, we notice that, if $y$ is a mild solution to \eqref{omo_equation}, then a straightforward consequence of Remark \ref{rmk:prop_int_young}(ii) is that
\begin{align*}
y(t)=S(t-\tau)y(\tau)+\mathscr I_{S,\sigma\circ y}(\tau,t), \qquad\;\, (\tau,t)\in[0,1]_<^2.  
\end{align*}

Since it is rather long, we divide the proof into some steps.

{\em Step 1}. 
We prove a general estimate of the right-hand side of \eqref{reg_sol-aaa}.

Taking advantage of Lemma \ref{lemma:int_conv_funz_Y_alpha_gamma} with $\gamma=\alpha-\theta$, it is not hard to prove that for every $y\in Y^{\alpha}_{\theta-\alpha}(0,T)$ the function
$[0,T]\ni t\mapsto S(t)\psi+\mathscr I_{S,\sigma\circ y}(0,t)$
belongs to $Y^{\alpha}_{\theta-\alpha}(0,T)$ for every $T\in (0,1]$.

We introduce the operator $\Gamma:Y^{\alpha}_{\theta-\alpha}(0,T)\to Y^{\alpha}_{\theta-\alpha}(0,T)$, defined by
\begin{align*}
(\Gamma(y))(t):=S(t)\psi+\mathscr I_{S,\sigma\circ y}(0,t), \qquad\;\, t\in[0,T].
\end{align*}
We claim that
\begin{align}
\|y\|_{Y^{\alpha}_{\theta-\alpha}(0,T)} 
\leq & (L_{0,0}K_{\theta,0}+L_{\theta,\alpha})\|\psi\|_{X_\theta}\notag \\
& +\mathfrak CT^{\eta+\theta-2\alpha}\|x\|_{C^\eta([0,1])}(1+\|y\|_{Y^{\alpha}_{\theta-\alpha}(0,T)}), 
\label{stima-strautile}
\end{align}
with
\begin{align*}
\mathfrak C:=C( K_{\alpha,0}+2)(K_{\alpha,0}+C_{\alpha,0}+1)({\rm Lip}_\sigma+L_\sigma^\alpha).     
\end{align*}
Recall that the constants $K_{\alpha,\beta}, L_{\zeta,\alpha}$ and $C_{\mu,\nu}$ have been defined in Hypotheses \ref{hyp-main}.

For every $s\in(0,T]$ it holds that
\begin{align}
& \|y(s)\|_X
\leq L_{0,0}K_{\theta,0}\|\psi\|_{X_\theta}+K_{\alpha,0}\|\mathscr I_{S,\sigma\circ y}(0,s)\|_{X_\alpha}, \label{stima_a_priori_1}\\
& \|y(s)\|_{X_\alpha}\leq L_{\theta,\alpha}s^{\theta-\alpha}\|\psi\|_{X_\theta}+\|\mathscr I_{S,\sigma\circ y}(0,s)\|_{X_\alpha}. \label{stima_a_priori_2}
\end{align}
We want to apply Remark \ref{lemma:def_int_conv} with $a=0, b=T$, $\varphi=\sigma\circ y$, $\gamma=\alpha-\theta$, $\rho=\alpha$ and $\beta=\varepsilon=0$. The continuity of $\sigma\circ y$ and of $\hat\delta_1(\sigma\circ y)$ follows from the properties of $y$ and $\sigma$. From \eqref{ip_sigma} we infer that
\begin{align*}
& \|\sigma\circ y\|_{\mathscr C_{\theta-\alpha}((0, T];X_\alpha)} \leq L_\sigma^\alpha(T^{\alpha-\theta}+\|y\|_{Y^{\alpha}_{\theta-\alpha}(0,T)}).   
\end{align*}
Let us estimate $\|\hat\delta_1(\sigma\circ y)\|_{\mathscr C^{\alpha}_{\theta-\alpha}((0,T]^2_<;X)}$. For every $(s,t)\in(0,T]$ we get
\begin{align}
\|(\hat\delta_1(\sigma\circ y))(s,t)\|_X
\leq & \|\sigma(y(t))-\sigma(y(s))\|_X+\|\mathfrak a(s,t)\sigma(y(s))\|_X \notag \\
\leq & {\rm Lip}_\sigma\|y(t)-y(s)\|_X+C_{\alpha,0}|t-s|^\alpha\|\sigma(y(s))\|_{X_\alpha} \notag \\
\leq & {\rm Lip}_\sigma(\|(\hat\delta_1y)(s,t)\|_X+\|\mathfrak a(s,t)y(s)\|_X) \notag \\
& +C_{\alpha,0}L_\sigma^\alpha(1+\|y(s)\|_{X_\alpha})|t-s|^\alpha \notag \\
\leq & \Big({\rm Lip}_\sigma(K_{\alpha,0}\|\hat\delta_1y\|_{C^{\alpha}([0,1]^2_<;X_\alpha)}+C_{\alpha,0}\|y(s)\|_{X_\alpha}) \notag \\
&\;\, +C_{\alpha,0}L_\sigma^\alpha(1+\|y(s)\|_{X_\alpha})\Big)|t-s|^\alpha,
\label{stima_a_priori_per_futuro_1}
\end{align}
from which it follows that
\begin{align*}
\|\hat\delta_1(\sigma\circ y)\|_{\mathscr C^{\alpha}_{\theta-\alpha}((0,T]^2_<;X)}
\! \leq  (K_{\alpha,0}+C_{\alpha,0})({\rm Lip}_\sigma+L_\sigma^\alpha)(1+\|y\|_{Y^{\alpha}_{\theta-\alpha}(0,T)}).
\end{align*}
From Remark \ref{lemma:def_int_conv} we deduce that there exists a positive constant $C$ such that
\begin{align}
& \|\mathscr I_{S,\sigma\circ y}\|_{\mathscr C^{\eta+\theta-\alpha}([0,T]^2_<;X_\alpha)} \notag \\
\leq &  C\|x\|_{C^\eta([0,1])}\Big[L_\sigma^\alpha( 1+\|y\|_{Y^{\alpha}_{\theta-\alpha}(0,T)}) \notag \\
&\phantom{C\|x\|_{C^{\eta}([0,1])}\Big [}+(K_{\alpha,0}+C_{\alpha,0})({\rm Lip}_\sigma+L_\sigma^\alpha)(1+\|y\|_{Y^{\alpha}_{\theta-\alpha}(0,T)})\Big] \notag \\
\leq & C\|x\|_{C^\eta([0,1])}(K_{\alpha,0}+C_{\alpha,0}+1)({\rm Lip}_\sigma+L_\sigma^\alpha)(1+\|y\|_{Y^{\alpha}_{\theta-\alpha}(0,T)}).
\label{stima_a_priori_3}
\end{align}
Replacing \eqref{stima_a_priori_3} in \eqref{stima_a_priori_1} and \eqref{stima_a_priori_2} we get
\begin{align}
& \|y\|_{C_b((0,T];X)}
+\|y\|_{\mathscr C_{\theta-\alpha}((0,T];X_\alpha)} \notag \\
\leq & (L_{0,0}K_{\theta,0}+L_{\theta,\alpha})\|\psi\|_{X_\theta}\notag\\ &+C(K_{\alpha,0}+1)T^{\eta+\theta-\alpha}\|x\|_{C^\eta([0,1])}(K_{\alpha,0}+C_{\alpha,0}+1)({\rm Lip}_\sigma+L_\sigma^\alpha)(1+\|y\|_{Y^{\alpha}_{\theta-\alpha}(0,T)}),
\label{stima_a_priori_tot_1}
\end{align}
where we used the fact that $T\le 1$ and $\alpha\geq \theta$ to estimate $T^{\eta}\le T^{\eta+\theta-\alpha}$.

It remains to estimate $\|\hat\delta_1y\|_{\mathscr C^\alpha([0,T]^2_<;X_{\alpha})}$. Since $(\hat\delta_1y)(s,t)=\mathscr I_{S,\sigma\circ y}(s,t)$ for every $(s,t)\in [0,T]$ and $\eta>2\alpha-\theta$, from \eqref{stima_a_priori_3} we infer that
\begin{align}
\|\hat\delta_1y\|_{\mathscr C^\alpha([0,T]^2_<;X_{\alpha})}
\leq & T^{\eta+\theta-2\alpha} \|\hat\delta_1y\|_{C^{\eta+\theta-\alpha}([0,T]^2_<;X_{\alpha})} \notag \\
\leq & T^{\eta+\theta-2\alpha}C\|x\|_{C^\eta([0,1])}(K_{\alpha,0}+C_{\alpha,0}+1)({\rm Lip}_\sigma+L_\sigma^\alpha) \notag  \\
&\qquad \times (1+\|y\|_{Y^{\alpha}_{\theta-\alpha}(0,T)}). \label{stima_a_priori_tot_2}
\end{align}
The definition of $\|\cdot\|_{Y^{\alpha}_{\theta-\alpha}(0,T)}$, \eqref{stima_a_priori_tot_1} and \eqref{stima_a_priori_tot_2} give \eqref{stima-strautile}.

{\em Step 2}. We are now in a position to prove a global {\em a priori} estimate.

We claim that, if $y$ is a mild solution to \eqref{omo_equation}, then there exists a positive constant $\mathfrak R$, which depends on $\|\psi\|_{X_{\theta}},\alpha,\theta,x,\eta$ and $\sigma$, such that
\begin{align}
\label{stima_a_priori}    
\|y\|_{Y^{\alpha}_{\theta-\alpha}(0,1)}\leq \mathfrak R. 
\end{align}
Notice that $y\in Y^{\alpha}_{\theta-\alpha}(0,1)$ is a mild solution to \eqref{omo_equation} if and only if it is a fixed point of $\Gamma$. Thus, if $y\in Y^{\alpha}_{\theta-\alpha}(0,1)$ is a mild solution and we choose
\begin{align*}
    \overline T:=\left(\frac{1}{2\mathfrak C\|x\|_{C^\eta([0,1])}}\right)^{\frac{1}{\eta+\theta-2\alpha}}\wedge 1,
\end{align*}
then \eqref{stima-strautile} immediately implies that
\begin{align*}
\|y\|_{Y^{\alpha}_{\theta-\alpha}(0,\overline T)}
\leq 2(L_{0,0}K_{\theta,0}+L_{\theta,\alpha})\|\psi\|_{X_\theta}+1.
\end{align*}
So, we can say that there exists $\overline T\in(0,1]$ such that $\|y\|_{Y^{\alpha}_{\theta-\alpha}(0,\overline T)}\leq \mathfrak R_1$ for some $\mathfrak R_1>0$ which depends on $\|\psi\|_{X_{\theta}},\alpha,\theta,x,\eta$ and $\sigma$. 

If $\overline T=1$, then we are done, otherwise we notice that $y\in C([\overline T,1];X_\alpha)$ and $\hat\delta_1y\in \mathscr{C}^{\alpha}([\overline T,1]_<^2;X_\alpha)$. In particular, $y$ is a solution in $[\overline T,1]$ with initial datum $y(\overline T)\in X_{\alpha}$ Arguing as in Step 1 in the proof of \cite[Theorem 3.1]{ALT}, we infer that there exists $\mathfrak R_2>0$, which depends on $\|y(\overline T)\|_{X_{\alpha}},\alpha,\theta,x,\eta$ and $\sigma$, such that
\begin{align*}
\|y\|_{C([\overline T,1];X_\alpha)}+\|\hat\delta_1y\|_{\mathscr C^{\alpha}([\overline T;1]^2_<;X_{\alpha})}\leq \mathfrak R_2.  \end{align*}
Moreover, $\|y(\overline T)\|_{X_{\alpha}}$ can be estimated by $\mathfrak{R}_1$ and $\overline T$ above, so that $\mathfrak{R}_2>0$ finally depends on $\|\psi\|_{X_{\theta}},\alpha,\theta,x,\eta$ and $\sigma$.

If we join the two estimates above, then we get
\begin{align*}
& \|y(s)\|_{X}
\leq \mathfrak R_1\vee (K_{\alpha,0}\mathfrak R_2), \qquad\;\, s\in(0,1], \\
& s^{\alpha-\theta}\|y(s)\|_{X_\alpha}
 \leq \mathfrak R_1\vee \mathfrak R_2, \qquad\;\, s\in (0,1], \\
& \|(\hat\delta_1y)(s,t)\|_{X_{\alpha}}
\leq (\mathfrak R_1\vee \mathfrak R_2)|t-s|^\alpha, \qquad\;\, (s,t)\in[0,\overline T] \textrm{ or } (s,t)\in[\overline T,1].
\end{align*}
If $s\in(0,\overline T]$ and $t\in(\overline T,1]$, recalling that 
\begin{align*}
(\hat\delta_1y)(s,t)=\mathscr I_{S,\sigma\circ y}(s,t)=S(\overline T-s)\mathscr I_{S,\sigma\circ y}(s,\overline T)+\mathscr I_{S,\sigma\circ y}(\overline T,t),
\end{align*}
then we get
\begin{align*}
\|(\hat\delta_1y)(s,t)\|_{X_{\alpha}}
\leq & L_{0,0}\|(\hat\delta_1y)(s,\overline T)\|_{X_{\alpha}}+\|(\hat\delta_1y)(\overline T,t)\|_{X_{\alpha}} \\
\leq & (L_{0,0}\mathfrak R_1|\overline T-s|^\alpha+\mathfrak R_2|t-\overline T|^\alpha) \\
\leq & (L_{0,0}\mathfrak R_1+\mathfrak R_2)|t-s|^\alpha.
\end{align*}
Putting everything together we infer that there exists a positive constant $\mathfrak R\geq \mathfrak R_1$, which depends on $\|\psi\|_{X_\theta},\alpha,\theta,x,\eta$ and $\sigma$, such that \eqref{stima_a_priori} holds true.

{\em Step 3}. Let us prove that $\Gamma$ is a contraction in the closed ball $B$ of $Y^{\alpha}_{\theta-\alpha}(0,T_*)$, centered at $0$ and with radius $\mathfrak R$, 
for some positive $T_*$, where $\mathfrak R$ is the constant in \eqref{stima_a_priori}. 
As a byproduct, we infer that there exists a unique mild solution $y_1$ to \eqref{omo_equation} in $[0,T_*]$. Indeed, if $\widetilde y$ is another mild solution to \eqref{omo_equation} in $[0,T_*]$, then from \eqref{stima_a_priori} it follows that $\|\widetilde y\|_{Y^{\alpha}_{\theta-\alpha}(0,T_*)}\leq \mathfrak R$, which means that $\widetilde y\in B$ and it is a fixed point of $\Gamma$. Hence, $\widetilde y=y_1$. 

Suppose that we have proved that $\Gamma$ is a contraction in 
$B$. If $T_*=1$ then we are done, otherwise we can apply the arguments in the proof of \cite[Theorem 3.1]{ALT} in $[T_*,1]$ with initial datum $y_1(T_*)\in X_\alpha$, exploiting the extra regularity of the initial datum. It follows that there exists a unique mild solution $y_2\in C([T_*,1];X_\alpha)$ with $\hat\delta_1 y_2\in \mathscr C^\alpha([T_*,1]^2_<;X_{\alpha})$ to the problem
\begin{align*}
\left\{
\begin{array}{ll}
dy(t)=Ay(t)dt+\sigma(y(t))dx(t), &t\in(T_*,1],\\[1mm]
y(T_*)=y_1(T_*).
\end{array}
\right.    
\end{align*}
Hence, if we set
\begin{align*}
y(t):=\begin{cases}
y_1(t), & t\in[0,T_*], \\
y_2(t), & t\in[T_*,1],
\end{cases}
\end{align*}
then we obtain the unique mild solution to \eqref{omo_equation}.

So, let us show that there exists $T_*\in(0,1]$ such that $\Gamma$ is a contraction in $B$.
We begin by proving that $\Gamma$ maps $B$ into itself. For this purpose, we fix $y_1\in B$. 
From \eqref{stima-strautile}, with $\overline T$ being replaced by $T^*$, it follows that
\begin{align*}
\|\Gamma(y)\|_{Y^{\alpha}_{\theta-\alpha}(0,T_*)}
\leq & (L_{0,0}K_{\theta,0}+L_{\theta,\alpha})\|\psi\|_{X_\theta} \\
& +\mathfrak C T_*^{\eta+\theta-2\alpha}\|x\|_{C^\eta([0,1])}(1+\|y\|_{Y^{\alpha}_{\theta-\alpha}(0,T_*)}) \\
\leq & 
\frac{1}{2}\mathfrak R +\mathfrak C T_*^{\eta+\theta-2\alpha}\|x\|_{C^\eta([0,1])}(1+\mathfrak R).
\end{align*}
By choosing
\begin{align*}
T_*= \left(\frac{1}{2}\frac{\mathfrak R}{\mathfrak C\|x\|_{C^\eta([0,1])}(1+\mathfrak R) }\right)^{\frac{1}{\eta+\theta-2\alpha}}\wedge 1,    
\end{align*}
it follows that
$\Gamma(y)\in B$.

Let us now prove that $\Gamma$ is a $1/2$-contraction, provided we choose $T_*$ small enough. Note that, for every $y_1,y_2\in B$, it holds that
\begin{align*}
(\Gamma(y_1))(t)-(\Gamma(y_2))(t)
= \mathscr I_{S,\sigma\circ y_1-\sigma\circ y_2}(0,t), \qquad\;\, t\in[0,T_*],
\end{align*}
and
\begin{align*}
(\hat\delta_1(\Gamma(y_1)-\Gamma(y_2)))(s,t)
= \mathscr I_{S,\sigma\circ y_1-\sigma\circ y_2}(s,t), \qquad\;\, (s,t)\in[0,T_*]^2_<.
\end{align*}
Since for every $(s,t)\in[0,T_*]^2_<$ it holds that
\begin{align*}
    \|\mathscr I_{S,\sigma\circ y_1-\sigma\circ y_2}(s,t)\|_{X_\alpha}
    \leq |t-s|^{\eta+\theta-\alpha}\|\mathscr I_{S,\sigma\circ y_1-\sigma\circ y_2}\|_{\mathscr C^{\eta+\theta-\alpha}([0,T_*]^2_<;X_\alpha)}
\end{align*}
using Remark \ref{lemma:def_int_conv}, with $\gamma=\alpha-\theta$, $a=0$, $b=T_*$, $\rho=\alpha$, $\beta=\varepsilon=0$, we deduce that
\begin{align}
\notag
\|\mathscr I_{S,\sigma\circ y_1-\sigma\circ y_2}(s,t)\|_{X_\alpha}
\leq & C|t-s|^{\eta+\theta-\alpha}\|x\|_{C^\eta([0,1])}(\|\sigma\circ y_1-\sigma\circ y_2\|_{\mathscr{C}_{\theta-\alpha}((0,T_*];X_\alpha)} \\
& +\|\hat \delta_1(\sigma\circ y_1-\sigma\circ y_2)\|_{\mathscr C^{\alpha}_{\theta-\alpha}((0,T_*]^2_<;X)}),
\label{stima_diff_contr_1}
\end{align}
for every $(s,t)\in[0,T_*]^2_<$. Let us estimate the first term in the right-hand side of \eqref{stima_diff_contr_1}. 
For every $s\in(0,T_*]$, we get
\begin{align}
\notag
 s^{\alpha-\theta}\|\sigma(y_1(s))-\sigma(y_2(s))\|_{X_\alpha}
 \leq & {\rm Lip}^{\alpha}_\sigma s^{\alpha-\theta}\|y_1(s)-y_2(s)\|_{X_\alpha} \\
 \leq & {\rm Lip}^{\alpha}_{\sigma}\|y_1-y_2\|_{Y^{\alpha}_{\theta-\alpha}(0,T_*)}. 
\label{stima_diff_contr_2}
 \end{align}
Next, we estimate $\|\hat \delta_1(\sigma\circ y_1-\sigma\circ y_2)\|_{\mathscr C^{\alpha}_{\theta-\alpha}((0,T_*]^2_<;X)}$. For every $(s,t)\in(0,T_*]^2_<$, we can write
\begin{align*}
(\hat \delta_1(\sigma\circ y_1-\sigma\circ y_2))(s,t)
= (\delta_1(\sigma\circ y_1-\sigma \circ y_2))(s,t)-\mathfrak a(s,t)(\sigma(y_1(s))-\sigma(y_2(s))),  \end{align*}
and
\begin{align}
\notag
s^{\alpha-\theta}\|\mathfrak a(s,t)(\sigma(y_1(s))-\sigma(y_2(s)))\|_X
\leq & C_{\alpha,0}{\rm Lip}_\sigma^\alpha s^{\alpha-\theta}\|y_1(s)-y_2(s)\|_{X_\alpha}|t-s|^\alpha \\
\leq & C_{\alpha,0}{\rm Lip}_\sigma^\alpha |t-s|^\alpha\|y_1-y_2\|_{Y^{\alpha}_{\theta-\alpha}(0,T_*)}.
\label{stima_diff_contr_3}
\end{align}
Further,
\begin{align}
& (\delta_1(\sigma\circ y_1-\sigma\circ y_2))(s,t)\notag \\
= & \sigma(y_1(s)+(\delta_1y_1)(s,t))-\sigma(y_1(s))-\big(\sigma(y_2(s)+(\delta_1y_1)(s,t))-\sigma(y_2(s))\big)\notag \\
& +\sigma(y_2(s)+(\delta_1y_1)(s,t))-\sigma(y_2(s)+(\delta_1y_2)(s,t))\notag \\
= & \int_0^1[ \sigma'(y_1(s)+r(\delta_1y_1)(s,t))-\sigma'(y_2(s)+r(\delta_1y_1)(s,t))](\delta_1y_1)(s,t)dr\notag \\
& +\sigma(y_2(s)+(\delta_1y_1)(s,t))-\sigma(y_2(s)+(\delta_1y_2)(s,t)),
\label{isolaman}
\end{align}
from which it follows that
\begin{align}
& s^{\alpha-\theta}\|\delta_1(\sigma\circ y_1-\sigma\circ y_2)(s,t)\|_{X} \notag\\
\leq & {\rm Lip}_{\sigma'}^{3\mathfrak R}s^{\alpha-\theta}\|y_1(s)-y_2(s)\|_X\|(\delta_1y_1)(s,t)\|_X+{\rm Lip}_{\sigma}s^{\alpha-\theta}\|(\delta_1y_1)(s,t)-(\delta_1y_2)(s,t))\|_{X} \notag\\
\leq & {\rm Lip}_{\sigma'}^{3\mathfrak R}s^{\alpha-\theta}\|y_1(s)-y_2(s)\|_X\big(\|(\hat\delta_1y_1)(s,t)\|_{X}+\|\mathfrak a(s,t)y_1(s)\|_X \big) \notag\\
& +{\rm Lip}_{\sigma}s^{\alpha-\theta}\big(\|(\hat\delta_1(y_1-y_2))(s,t)\|_X+\|\mathfrak a(s,t)(y_1(s)-y_2(s))\|_X\big) \notag \\
\leq & {\rm Lip}_{\sigma'}^{3\mathfrak R}\|y_1-y_2\|_{Y^{\alpha}_{\theta-\alpha}(0,T_*)}\big(K_{\alpha,0}\|y_1\|_{Y^{\alpha}_{\theta-\alpha}(0,T_*)}+C_{\alpha,0}s^{\alpha-\theta}\|y_1(s)\|_{X_\alpha}\big)|t-s|^\alpha \notag \\
& +{\rm Lip}_{\sigma}\big(K_{\alpha,0}\|y_1-y_2\|_{Y^{\alpha}_{\theta-\alpha}(0,T_*)}+C_{\alpha,0}s^{\alpha-\theta}\|y_1(s)-y_2(s)\|_{X_\alpha}\big)|t-s|^\alpha \notag \\
\leq & {\rm Lip}_{\sigma'}^{3\mathfrak R}(K_{\alpha,0}+C_{\alpha,0})\|y_1\|_{Y^{\alpha}_{\theta-\alpha}(0,T_*)}\|y_1-y_2\|_{Y^{\alpha}_{\theta-\alpha}(0,T_*)}|t-s|^\alpha \notag \\
& + {\rm Lip}_{\sigma}(K_{\alpha,0}+C_{\alpha,0})\|y_1-y_2\|_{Y^{\alpha}_{\theta-\alpha}(0,T_*)}|t-s|^\alpha.
\label{stima_diff_contr_4}
\end{align}
Putting together \eqref{stima_diff_contr_1}-\eqref{stima_diff_contr_4}, we infer that
\begin{align}
\|\mathscr I_{S,\sigma\circ y_1-\sigma\circ y_2}(s,t)\|_{X_\alpha}
\leq & \widetilde C |t-s|^{\eta+\theta-\alpha}\|x\|_{C^\eta([0,1])}\|y_1-y_2\|_{Y^{\alpha}_{\theta-\alpha}(0,T_*)},
\label{AA}
\end{align}
for every $(s,t)\in[0,T_*]^2_<$ and some positive constant $\widetilde C$
which depends on $\alpha$, $\theta$, $\eta$, $\mathfrak R$ and $\sigma$.
Further,
\begin{align}
\frac{\|\mathscr I_{S,\sigma\circ y_1-\sigma\circ y_2}(s,t)\|_{X_{\alpha}}}{|t-s|^\alpha}
\leq & \frac{\|\mathscr I_{S,\sigma\circ y_1-\sigma\circ y_2}(s,t)\|_{X_{\alpha}}}{|t-s|^{\eta-\theta-\alpha}}|t-s|^{\eta+\theta-2\alpha}\notag \\
\leq & \widetilde C T_*^{\eta+\theta-2\alpha}\|x\|_{C^\eta([0,1])}\|y_1-y_2\|_{Y^{\alpha}_{\theta-\alpha}(0,T_*)}
\label{BB}
\end{align}
for every $(s,t)\in[0,T_*]^2_<$. Therefore, we get
\begin{align*}
& t^{\alpha-\theta}\|(\Gamma( y_1))(t)-(\Gamma (y_2))(t)\|_{X_\alpha}+\|(\Gamma (y_1))(t)-(\Gamma (y_2))(t)\|_X\\
&+\frac{\|(\hat\delta_1(\Gamma (y_1)-\Gamma (y_2)))(s,t)\|_{X_{\alpha}}}{|t-s|^\alpha} \\
\le &\|\mathscr I_{S,\sigma\circ y_1-\sigma\circ y_2}\|_{X_\alpha}+
\|\mathscr I_{S,\sigma\circ y_1-\sigma\circ y_2}\|_X
+\frac{\|\mathscr I_{S,\sigma\circ y_1-\sigma\circ y_2}(s,t)\|_{X_{\alpha}}}{|t-s|^\alpha} \\
\leq & \widetilde CT_*^{\eta+\theta-2\alpha}(K_{\alpha,0}+2)
\|x\|_{C^\eta([0,1])}\|y_1-y_2\|_{Y^{\alpha}_{\theta-\alpha}(0,T_*)}
\end{align*}
for every $(s,t)\in[0,T_*]^2_<$. This implies that a suitable choice of $T_*$ gives
\begin{align*}
\|\Gamma(y_1)-\Gamma(y_2)\|_{Y^{\alpha}_{\theta-\alpha}(0,T_*)}\leq \frac12\|y_1-y_2\|_{Y^{\alpha}_{\theta-\alpha}(0,T_*)},  \end{align*}
i.e., $\Gamma$ is a $\frac12$-contraction on $B$ and, therefore, it admits a unique fixed point in $B$ which we denote by $y_1$.
\end{proof}

Now, we prove some regularizing properties of the solution $y$ to \eqref{omo_equation}. 

\begin{proposition}
\label{prop:reg_mild_sol}
Let Hypotheses $\ref{ip_nonlineare}$ be satisfied. Then, for every $\psi\in X_\theta$ with $\alpha\geq \theta$ and $\eta>2\alpha-\theta$, the unique mild solution $y$ to \eqref{omo_equation} in $[0,1]$ belongs to $C((0,1];X_\rho)$ and $C^{\eta+\alpha-\rho}_{\rm loc}((0,1];X_\rho)$ for every $\rho \in[\eta+\alpha-1,\eta+\alpha)$. Finally, for every $\mu\in[0,\eta+\alpha-1)$ there exists a positive constant $\mathfrak C=\mathfrak C(\|\psi\|_{X_\theta},\alpha,\theta,x,\eta,\sigma,\mu)$ such that
\begin{align}
\label{stima_sol_mild_DA}
\|y(t)\|_{X_{1+\mu}}\leq \mathfrak C t^{\theta-1-\mu}, \qquad\;\, t\in(0,1].
\end{align}
\end{proposition}

\begin{remark}
{\rm We stress that the behaviour of the $X_{1+\mu}$-norm of $y$ in estimate \eqref{stima_sol_mild_DA} is sharp. Indeed, in the particular case when $\sigma\equiv 0$, $y(t)=S(t)\psi$ for every $t\in (0,1]$ and estimate \eqref{stima_sol_mild_DA} agrees with \eqref{stime_smgr}(a).}
\end{remark}

\begin{proof}[Proof of Proposition $\ref{prop:reg_mild_sol}$]
Let us notice that, for every $\tau\in(0,1)$, $y$ is the unique mild solution to
\begin{align*}
\left\{
\begin{array}{ll}
dv(t)=Av(t)dt+\sigma(v(t))dx(t), &t\in(\tau,1],\\[1mm]
v(\tau)=y(\tau),
\end{array}
\right.   
\end{align*}
with $y(\tau)\in X_\alpha$. From \cite[Theorem 3.1]{ALT}, it follows that $y(t)\in C((0,1];X_\rho)$ and $C^{\eta+\alpha-\rho}_{\rm loc}((0,1];X_\rho)$ for every $\rho \in[\eta+\alpha-1,\eta+\alpha)$. It remains to prove estimate \eqref{stima_sol_mild_DA}. 

Let us fix $\tau\in(0,1]$ and observe that, using Remark \ref{rmk:prop_int_young} and formula \eqref{reg_sol-aaa}, we can easily show that
\begin{align}
y(t)=S(t-\tau/2)y(\tau/2)+\mathscr I_{S,\sigma\circ y}(\tau/2,t)
\label{brasile-ghana}
\end{align}
for every $t\in[\tau,1]$, so that, using \eqref{stima-utilee}, which holds true also when $\gamma=0$ (see Remark \ref{rmk-abbruzzetti}), we obtain
\begin{align}
\|y(t)\|_{X_\zeta}
\leq & L_{\alpha,\zeta}(t-\tau/2)^{\alpha-\zeta}\|y(\tau/2)\|_{X_\alpha}+C(t-\tau/2)^{\eta+\alpha-\zeta}\|x\|_{C^\eta([0,1])} \notag \\
& \times (\|\sigma\circ y\|_{C([\tau/2,1];X_\alpha)}
+\|\hat\delta_1 (\sigma\circ y)\|_{\mathscr{C
}^\alpha([\tau/2,1]^2_<;X)}) \notag \\
\leq & 2^{\alpha-\theta}L_{\alpha,\zeta,1}\|y\|_{Y^{\alpha}_{\theta-\alpha}(0,1)}(t-\tau/2)^{\alpha-\zeta}\tau^{\theta-\alpha}
+ C(t-\tau/2)^{\eta+\alpha-\zeta}\|x\|_{C^\eta([0,1])} \notag \\
& \times (\|\sigma\circ y\|_{C([\tau/2,1];X_\alpha)}
+\|\hat\delta_1 (\sigma\circ y)\|_{\mathscr{C}^{\alpha}([\tau/2,1]^2_<;X)})
\label{stima_y_reg_DA_0}
\end{align}
for every $\zeta\in[\alpha,1)$ and $t\in [\tau,1]$.

Let us estimate the last factor in the right-hand side above: from \eqref{ip_sigma-1}, with $\gamma=\alpha-\theta$, we get
\begin{align}
\|\sigma\circ y\|_{C([\tau/2,1];X_\alpha)}\leq 2^{\alpha-\theta}L_\sigma^\alpha \tau^{\theta-\alpha}(1+\|y\|_{Y^{\alpha}_{\theta-\alpha}(0,1)}).
\label{stima_y_reg_DA_1}
\end{align}
By taking advantage of \eqref{stima_a_priori_per_futuro_1}, we easily infer that 
\begin{align}
& \|\hat\delta_1(\sigma\circ y)\|_{\mathscr{C}^{\alpha}([\tau/2,1]^2_<;X)} \notag \\
\leq &  2^{\alpha-\theta}\tau^{\theta-\alpha}(K_{\alpha,0}+C_{\alpha,0})({\rm Lip}_\sigma+L_\sigma^\alpha)(1+\|y\|_{Y^{\alpha}_{\theta-\alpha}(0,1)}). \label{stima_y_reg_DA_2}
\end{align}
By replacing \eqref{stima_y_reg_DA_1} and \eqref{stima_y_reg_DA_2} in \eqref{stima_y_reg_DA_0} we conclude that
\begin{align}
\label{stima_reg_y_DAfinale_1}
\|y(t)\|_{X_\zeta}
\leq c_1(t-\tau/2)^{\alpha-\zeta}\tau^{\theta-\alpha}, \qquad\;\, t\in[\tau,1],
\end{align}
for some positive constant $c_1$ which depends on  $\|\psi\|_{X_\theta},\alpha,\theta,x,\eta,\sigma$ and $\zeta$.

Now, we need to go beyond $\zeta<1$. To this aim, we fix $\lambda\in[0,\eta+\alpha-1)$ and we estimate $\|(\hat\delta_1(\sigma\circ y))(s,t)\|_{X_\lambda}$. Since $\eta<1$ it follows that $\lambda<\alpha$. For every $(s,t)\in[\tau,1]^2_<$ we get
\begin{align*}
\|(\hat\delta_1(\sigma\circ y))(s,t)\|_{X_\lambda}
\leq & \|\sigma(y(t))-\sigma(y(s))\|_{X_\lambda}+\|\mathfrak a(s,t)\sigma(y(s))\|_{X_\lambda} \notag\\
\leq & K_{\alpha,\lambda}{\rm Lip}_\sigma^\alpha\|y(t)-y(s)\|_{X_\alpha} +C_{\alpha,\lambda}L_{\sigma}^\alpha(1+\|y(s)\|_{X_{\alpha}})|t-s|^{\alpha-\lambda} \notag\\
\leq & K_{\alpha,\lambda}{\rm Lip}_\sigma^\alpha(\|(\hat\delta_1y)(s,t)\|_{X_\alpha}+\|\mathfrak a(s,t)y(s)\|_{X_\alpha}) \notag \\
& +C_{\alpha,\lambda}L_{\sigma}^{\alpha}\tau^{\theta-\alpha}(1+\|y\|_{Y^{\alpha}_{\theta-\alpha}(0,1)})|t-s|^{\alpha-\lambda}.
\end{align*}

To estimate the first term in the last side of the previous chain of inequalities, we observe that
\begin{align*}
&\|(\hat\delta_1y)(s,t)\|_{X_\alpha}+\|\mathfrak a(s,t)y(s)\|_{X_\alpha}\\
\leq & \|y\|_{Y^{\alpha}_{\theta-\alpha}(0,1)}|t-s|^{\alpha}+C_{2\alpha-\lambda,\alpha}\|y(s)\|_{X_{2\alpha-\lambda}}|t-s|^{\alpha-\lambda} \notag \\
\leq & |t-s|^{\alpha-\lambda}(\|y\|_{Y^{\alpha}_{\theta-\alpha}(0,1)}+c_1C_{2\alpha-\lambda,\alpha}(s-\tau/2)^{\lambda-\alpha}\tau^{\theta-\alpha} )\notag \\
\leq & (\|y\|_{Y^{\alpha}_{\theta-\alpha}(0,1)}+2^{\alpha-\lambda}c_1C_{2\alpha-\lambda,\alpha})\tau^{\lambda+\theta-2\alpha}|t-s|^{\alpha-\lambda},
\end{align*}
where we have applied \eqref{stima_reg_y_DAfinale_1} with $\zeta=2\alpha-\lambda$, so that
\begin{align}
\|(\hat\delta_1(\sigma\circ y))(s,t)\|_{X_\lambda}\le &
K_{\alpha,\lambda}{\rm Lip}^{\alpha}_{\sigma}
(\|y\|_{Y^{\alpha}_{\theta-\alpha}(0,1)}+2^{\alpha-\lambda}c_1C_{2\alpha-\lambda,\alpha})\tau^{\lambda+\theta-2\alpha}|t-s|^{\alpha-\lambda}\notag\\
&+C_{\alpha,\lambda,}L_{\sigma}^{\alpha}\tau^{\theta-\alpha}(1+\|y\|_{Y^{\alpha}_{\theta-\alpha}(0,1)})|t-s|^{\alpha-\lambda}
\label{stima_reg_delta_y_1}.
\end{align}
It follows that
\begin{align}
\label{stima_reg_delta_cap_y}
\|\hat\delta_1(\sigma\circ y)\|_{\mathscr C^{\alpha-\lambda}([\tau,1]^2_<;X_\lambda)}
\leq c_2\tau^{\lambda+\theta-2\alpha},
\end{align}
where $c_2$ is a positive constant which depends on $\|\psi\|_{X_\theta},\alpha,\theta,x,\eta,\sigma$ and $\lambda$. 

Finally, we fix $\mu\in[0,\eta+\alpha-1)$ and take
\begin{align*}
\lambda=(\eta+\alpha-1)-\frac{1}{2}[(\eta+\alpha-1-\mu)\wedge(\eta-\alpha+\mu)],
\end{align*}
which belongs to the interval $(\mu,\eta+\alpha-1)$.
Applying again
estimate \eqref{stima-utilee}, with
$\beta=\lambda$, $\rho=\alpha-\lambda$, $\gamma=0$
and $\varepsilon=1+\mu-\alpha$, we infer that $\mathscr I_{S,\sigma\circ y}\in \mathscr C^{\eta+\alpha-1-\mu}([\varepsilon,1]^2_<;X_{1+\mu})$ and
\begin{align}
\|\mathscr I_{S,\sigma\circ y}\|_{\mathscr C^{\eta+\alpha-1-\mu}([\tau,1]^2_<;X_{1+\mu})}
\leq & C(\|\sigma\circ y\|_{C([\tau,1];X_\alpha)}+\|\hat\delta_1(\sigma\circ y)\|_{\mathscr{C}^{\alpha-\lambda}([\tau,1]^2_<;X_{\lambda})}) \notag \\
\leq & C(L_\sigma^\alpha \tau^{\theta-\alpha}(1+\|y\|_{Y^{\alpha}_{\theta-\alpha}(0,1)}) +c_2\tau^{\theta+\lambda-2\alpha}) \notag \\
\leq & c_3\tau^{\theta+\lambda-2\alpha},
\label{stima_int_reg_1}
\end{align}
where $c_3$ is a positive constant which depends on $\|\psi\|_{X_\theta},\alpha,\theta,x,\eta,\sigma,\lambda$ and $\mu$. 

Taking \eqref{brasile-ghana} into account and applying \eqref{stima_int_reg_1}, with $\tau$ replaced by $\tau/2$, we deduce that
\begin{align}
\|y(t)\|_{X_{1+\mu}}    
\leq & L_{1+\mu,\alpha}|t-\tau/2|^{\alpha-1-\mu}\|y(\tau/2)\|_{X_\alpha}+2^{2\alpha-\theta-\lambda}c_3\tau^{\theta+\lambda-2\alpha}\notag\\
\leq & L_{1+\mu,\alpha}|t-\tau/2|^{\alpha-1-\mu}\tau^{\theta-\alpha}\|y\|_{Y^{\alpha}_{\theta-\alpha}(0,1)}+2^{2\alpha-\theta-\lambda}c_3\tau^{\theta+\lambda-2\alpha} \notag\\
\leq & c_4\tau^{\theta-1-\mu}
\label{macau-2}
\end{align}
for every $t\in[\tau,1]$, where $c_4$ is a positive constant which depends on $\|\psi\|_{X_\theta}$, $\alpha$, $\theta$, $x$, $\eta$, $\sigma$ and $\mu$. 
In particular, we get $\|y(t)\|_{X_{1+\mu}}\leq c_4t^{\theta-1-\mu}$ for every $t\in(0,1]$, where $c_4$ is independent of $t$,
since $\lambda>\eta+\alpha-1-(\eta-\alpha-\mu)/2$, so that
\begin{align*}
\lambda-2\alpha+1+\mu
\geq 
\frac{1}{2}(\eta-\alpha+\mu)>0,
\end{align*}
due to the condition $\eta>2\alpha-\theta$, and this yields the inequality $\theta+\lambda-2\alpha\geq \theta-1-\mu$.
\end{proof}

\begin{remark}
\label{rmk-finale-mond}
{\rm Hypotheses \ref{hyp-main} is assumed with $0\le\zeta\le\lambda< 2$, since Proposition \ref{prop:reg_mild_sol} involves only the spaces $X_{\gamma}$, with $\gamma<2$. It is easy to check that, if Hypotheses \ref{hyp-main} are satisfied for every $\lambda$ and $\zeta$ such that $0\le\zeta\le\lambda\le 1+\beta$ for some $\beta\in (0,1)$, then  Proposition  \ref{prop:reg_mild_sol} still holds true with $\rho$ and $\mu$ replaced by $\rho\wedge (1+\beta)$ and $1+\mu\wedge\beta$, respectively.}
\end{remark}

\subsection{The case when $\sigma$ is locally Lipschitz continuous in $X_\alpha$}
\label{subsect-3.2}

In this subsection, we prove that a mild solution to \eqref{omo_equation} also exists if we weaken the assumptions on $\sigma$ as long as we strengthen the hypotheses on the H\"older exponent $\eta$ of $x$. As in the previous subsection, we get global existence and uniqueness of the mild solution $y$ and we provide regularity properties of $y$.

\begin{hypotheses}
\label{hyp:non_lineare_2}
\begin{enumerate}[\rm (i)]
\item
Hypotheses $\ref{hyp-main}$
are satisfied with $0\le\zeta\le\lambda<2$; 
\item
Hypotheses
$\ref{ip_nonlineare}(i)$-$(ii)$ are satisfied;
\item
there exists $\alpha\in (0,1)$ such that $\eta+\alpha>1$ and the restriction of $\sigma$ to $X_\alpha$ maps the space into itself. Moreover, there exist positive constants $L_\sigma^\alpha$, ${\rm Lip}_\sigma^\alpha$ and $\omega$ such that 
\begin{align}
\label{loc_lip_sigma_ip}
\;\;\;\;\;\;\;\;\;\|\sigma(x)-\sigma(y)\|_{X_{\alpha}}\leq {\rm Lip}_\sigma^\alpha (1+R)^{\omega}\|x-y\|_{X_\alpha}, \qquad\;\, x,y\in \overline{B(0,R)}\subset X_\alpha,
\end{align}
for every $R>0$ and 
$\|\sigma(x)\|_{X_\alpha}\leq L_\sigma^\alpha(1+\|x\|_{X_\alpha})$ for every $x\in X_\alpha$.
\end{enumerate}
\end{hypotheses}

\begin{theorem}
\label{thm:ex_mild_sol_2}
Let Hypotheses
$\ref{hyp:non_lineare_2}$ be satisfied.
Then, for every $\psi\in X_\theta$ with $\alpha\geq \theta$ and $\eta>\alpha+(1+\omega)(\alpha-\theta)$, there exists a unique mild solution $y$ to \eqref{omo_equation} which belongs to $Y^{\alpha}_{\theta-\alpha}(0,1)$. 
\end{theorem}

\begin{proof}[ Proof of Theorem $\ref{thm:ex_mild_sol_2}$]
Let us notice that, under Hypotheses \ref{hyp:non_lineare_2}, the estimates in the proof of Theorem \ref{thm:ex_mild_sol_1} which fail are \eqref{stima_diff_contr_2} and \eqref{stima_diff_contr_3}. Indeed, if $y_1,y_2$ belong to the ball $B\subset Y^{\alpha}_{\theta-\alpha}(0,T_*)$ with radius $\mathfrak R$, then $\|y_1(s)\|_{X_\alpha},\|y_2(s)\|_{X_\alpha}\leq s^{\theta-\alpha}\mathfrak R$ for every $s\in(0,T_*]$. From \eqref{loc_lip_sigma_ip} we get
\begin{align}
s^{\alpha-\theta}\|\sigma(y_1(s))-\sigma(y_2(s))\|_{X_\alpha}
\leq & {\rm Lip}_\sigma^\alpha (1+\mathfrak R)^\omega s^{-\omega(\alpha-\theta)} \|y_1-y_2\|_{Y^{\alpha}_{\theta-\alpha}(0,T_*)}
\label{stima_non_funz_loc_lips}
\end{align}
and, if $\omega>0$ and we take the supremum of $s$ over $(0,T_*)$, then the right-hand side of \eqref{stima_non_funz_loc_lips}  blows up. To overcome this problem, we stress that under Hypotheses \ref{hyp:non_lineare_2} we are able to apply Remark \ref{lemma:def_int_conv}, with $a=0$, $b=T_*$, $\varphi=\sigma\circ y_1-\sigma\circ y_2$, $\gamma=(1+\omega)(\alpha-\theta)$, $\rho=\alpha$ and $\beta=\varepsilon=0$. In this case, for every $t\in[0,T_*]$, instead of \eqref{stima_diff_contr_1}  we get
\begin{align}
\|\mathscr I_{S,\sigma\circ y_1-\sigma\circ y_2}(s,t)\|_{X_\alpha}
\leq & C|t-s|^{\eta-(1+\omega)(\alpha-\theta)}\|x\|_{C^\eta([0,1])}\notag \\
&\quad \times (\|\sigma\circ y_1-\sigma\circ y_2\|_{\mathscr C_{-(1+\omega)(\alpha-\theta)}((0,T_*];X_\alpha)}\notag \\
&\qquad\quad +\|\hat \delta_1(\sigma\circ y_1-\sigma\circ y_2)\|_{\mathscr C^{\alpha}_{-(1+\omega)(\alpha-\theta)}((0,T_*]^2_<;X)})
\label{macau}
\end{align}
for every $(s,t)\in [0,~T]^2_{<}$. Estimate \eqref{stima_non_funz_loc_lips} shows that $\sigma\circ y_1-\sigma\circ y_2$ belongs to $\mathscr{C}_{-(1+\omega)(\alpha-\beta)((0,T_*];X_{\alpha})}$ and
\begin{align}
\|\sigma\circ y_1-\sigma\circ y_2\|_{\mathscr{C}_{-(1+\omega)(\alpha-\beta)}((0,T_*];X_\alpha)}
 \leq & {\rm Lip}_\sigma^\alpha (1+\mathfrak R)^\omega \|y_1-y_2\|_{Y^{\alpha}_{\theta-\alpha}(0,T_*)}, 
 \label{stima_diff_lambda_1}
\end{align}
from which it follows immediately that
\begin{align}
& s^{(1+\omega)(\alpha-\theta)}
\|\mathfrak a(s,t)(\sigma(y_1(s))-\sigma(y_2(s))\|_{X} \notag \\
\leq &C_{\alpha,0}{\rm Lip}_{\sigma}^\alpha(1+\mathfrak R)^\omega|t-s|^\alpha\|y_1-y_2\|_{Y^{\alpha}_{\theta-\alpha}(0,T_*)}.
 \label{stima_diff_lambda_2}
\end{align}
Moreover, we can apply  \eqref{stima_diff_contr_4} (which does not rely on the Lipschitzianity of $\sigma$ in $X_\alpha$), since
$s^{(1+\omega)(\alpha-\theta)}\|(\delta_1(\sigma\circ y_1-\sigma\circ y_2))(s,t)\|_{X}$
and from \eqref{stima_diff_lambda_2} we infer that $\hat\delta_1(\sigma\circ y_1-\sigma\circ y_2)\in\mathscr{C}^{\alpha}_{-(1+\omega)(\alpha-\theta)}((0,T_*]^2_{<};X)$ and
\begin{align}
&\|\hat\delta_1(\sigma\circ y_1-\sigma\circ y_2)\|_{\mathscr{C}^{\alpha}_{-(1+\omega)(\alpha-\theta)}((0,T_*]^2_{<};X)}\notag\\
\le &[C_{\alpha,0}
{\rm Lip}_{\sigma}^\alpha(1+\mathfrak R)^\omega
+(K_{\alpha,0}+C_{\alpha,0})({\rm Lip}_{\sigma'}^{3\mathfrak R}\mathfrak{R}+ {\rm Lip}_{\sigma})]\|y_1-y_2\|_{Y^{\alpha}_{\theta-\alpha}(0,T_*)}.
\label{macau-1}
\end{align}

Replacing estimates \eqref{stima_diff_lambda_1} and \eqref{macau-1} into \eqref{macau}, we conclude that
\begin{align*}
&\|\mathscr I_{S,\sigma\circ y_1-\sigma\circ y_2,}(0,t)\|_{X_\alpha}
\leq \overline C T_*^{\eta-(1+\omega)(\alpha-\theta)}\|x\|_{C^\eta([0,1])}\|y_1-y_2\|_{Y^{\alpha}_{\theta-\alpha}(0,T_*)}, \\[1mm]
&\frac{\|\mathscr I_{S,\sigma\circ y_1-\sigma\circ y_2}(s,t)\|_{X_{\alpha}}}{|t-s|^\alpha}
\leq \overline C T_*^{\eta-(1+\omega)(\alpha-\theta)-\alpha}\|x\|_{C^\eta([0,1])}\|y_1-y_2\|_{Y^{\alpha}_{\theta-\alpha}(0,T_*)}
\end{align*}
for every $t\in[0,T_*]$ and $(s,t)\in[0,T_*]^2_{<}$, respectively. Here, $\overline C$ is a positive constant which depends on $\alpha,\theta,x$, $\eta$, $\mathfrak R$, $\sigma$ and $\omega$. Using these estimates, which replace \eqref{AA} and \eqref{BB}, and arguing as in the proof of Theorem \ref{thm:ex_mild_sol_1}, we can complete the proof.
\end{proof}

Under the same assumptions of Theorem \ref{thm:ex_mild_sol_2} we show that the mild solution $y$ is indeed more regular.
\begin{proposition}
\label{prop:reg_mild_sol_2}
Let Hypotheses $\ref{hyp:non_lineare_2}$ be satisfied. Then, for every $\psi\in X_\theta$ with $\alpha\geq \theta$ and $\eta>\alpha+(1+\omega)(\alpha-\theta)$, the unique mild solution $y$ to \eqref{omo_equation} in $[0,1]$ belongs to $C((0,1];X_\rho)$ and $C_{\rm loc}^{\eta+\alpha-\rho}((0,1];X_\rho)$ for every $\rho \in[\eta+\alpha-1,\eta+\alpha)$. Finally, for every $\mu\in[0,\eta+\alpha-1)$ there exists a positive constant $\mathfrak C=\mathfrak C(\|\psi\|_{X_\theta},\alpha,\theta,x,\eta,\sigma,\mu,\omega)$ such that
\eqref{stima_sol_mild_DA} holds true.
\end{proposition}

\begin{proof}
The proof can be obtained arguing as in the proof of Proposition \ref{prop:reg_mild_sol}, with the unique difference that, due to the fact that $\sigma$ is only locally Lipschitz continuous on $X_\alpha$, in estimate
\eqref{stima_reg_delta_y_1} we need to replace
${\rm Lip}_{\sigma}^{\alpha}$ with ${\rm Lip}_{\sigma}^\alpha\|y\|_{Y^\alpha_{\theta-\alpha}(0,1)}^\omega \tau^{\omega(\theta-\alpha)}$, so that
\eqref{stima_reg_delta_cap_y} becomes $\|\hat\delta_1(\sigma\circ y)\|_{\mathscr C^{\alpha-\lambda}([\tau,1];^2_<;X_\lambda)}\leq c_1\tau^{\lambda+\theta-2\alpha-\omega(\alpha-\theta)}$, with $\lambda\in[0,\eta+\alpha-1)$.
Here, $c_1$ is a positive constant
which depends on $\|\psi\|_{X_\theta}$, $\alpha$, $\theta$, $x$, $\eta$, $\sigma$ and $\lambda$.

Now, we fix $\mu\in [0,\eta+\alpha-1)$ and argue as in \eqref{macau-2} to infer that
\begin{align*}
\|y(t)\|_{X_{1+\mu}}    
\leq & 2^{\mu+1-\alpha}L_{1+\mu,\alpha}\tau^{\theta-\mu-1}\|y\|_{Y^{\alpha}_{\theta-\alpha}(0,1)}\\
& +2^{2\alpha+\omega(\alpha-\theta)-\theta-\lambda}c_2\tau^{\theta+\lambda-2\alpha-\omega(\alpha-\theta)} 
\end{align*}
for every $t\in[\tau,1]$, where $c_2$ is a positive constant which depends on $\mu$ and the same parameters as $c_1$.

Finally, we choose a suitable $\lambda$ such that $\tau^{\theta-\mu-1}\ge\tau^{\theta+\lambda-2\alpha-\omega(\alpha-\theta)}$ for every $\tau\in (0,1)$. For instance, we can take  
\begin{align*}
\lambda=\eta+\alpha-1-\bigg [\frac12(\eta+\alpha-1-\mu)\bigg ]\wedge(\eta-\alpha-\omega(\alpha-\theta)+\mu),    
\end{align*}
which belongs to the interval $(\mu,\eta+\alpha-1)$. The proof is complete.
\end{proof}

\begin{remark}
{\rm Clearly, Remark \ref{rmk-finale-mond} can be applied also to the results of this subsection.}
\end{remark}

\begin{example}
\label{example-3.14}
{\rm Let $\mathcal{A}$ be the second-order elliptic operator on $\mathbb R^d$ defined by
\begin{eqnarray*}
{\mathcal A}=\sum_{i,j=1}^dq_{ij}D_{ij}+\sum_{j=1}^db_jD_j+c.
\end{eqnarray*}
We assume that the coefficients of operator $\mathcal A$ are bounded and $\beta$-H\"older continuous in $\mathbb R^d$, for some $\beta\in (0,1)$, and $\sum_{i,j=1}^dq_{ij}(x)\xi_i\xi_j\ge\mu |\xi|^2$
for every $x,\xi\in\R^d$ and some positive constant $\mu$. 

Let $A$ be the realization of $\mathcal A$ in $X=C_b(\R^d)$ with maximal domain 
\begin{eqnarray*}
D(A)=\Big\{u\in C_b(\R^d)\cap\bigcap_{p<\infty}W^{2,p}_{\rm loc}(\R^d): {\mathcal A} u\in C_b(\R^d)\Big\}.
\end{eqnarray*}
For every $\lambda\in (0,1+\beta/2)\setminus\{1/2,1\}$, we take $X_{\lambda}=C^{2\lambda}_b(\R^d)$ endowed with its classical norm.
Moreover, we take as $X_{1/2}$ the Zygmund space of all bounded functions $g:\R^d\to\R$ such that $[g]_{X_{1/2}}=\displaystyle\sup_{x\neq y}\frac{|g(x)+g(y)-2g(2^{-1}(x+y))|}{|x-y|}<\infty$ endowed with the norm $\|g\|_{X_{1/2}}=\|g\|_{\infty}+[g]_{X_{1/2}}$. Finally, we set $X_1=D(A)$.

The operator $A$ generates an analytic semigroup on $C_b(\R^d)$. Further, for every $\lambda\in (0,1+\beta/2)$, $X_{\lambda}$ is the interpolation space of indexes $\lambda$ and $\infty$ between $X$ and $D(A)$. Hence, Hypotheses \ref{hyp:non_lineare_2} is satisfied for every $0\le\zeta\le\lambda\le 1+\beta$. We refer the reader to e.g., \cite[Chapters 3 and 14]{lor-rha}.
Finally, we fix a function $\hat\sigma\in C^2_b(\R)$ and note that the function $\sigma:X\to X$, defined by $\sigma(f)=\hat\sigma\circ f$ for every $f\in X$, satisfies Hypotheses \ref{hyp:non_lineare_2} for every $\alpha\in (0,1/2)$ with $\omega =1$.

Therefore, the assumptions of Theorem \ref{thm:ex_mild_sol_1} and Proposition \ref{prop:reg_mild_sol} are satisfied and we conclude that, for every $\psi\in C^{\theta}_b(\R^d)$ with $\theta\geq 0$ and $\theta\in ((3\alpha-\eta)/2,\alpha]$, there exists a unique mild solution $y$ to problem \eqref{omo_equation} which takes values in $D(A)$. In particular, if $3\alpha<1$ and $\eta>3\alpha$, then we get existence and uniqueness of the smooth mild solution $y$ for every $\psi\in X$. We notice that, since we are assuming $\eta+\alpha>1$, the inequality $\eta>3\alpha$ implies $\eta>3/4$.

Similarly, if $\alpha\in (1/2,1)$ and $\hat\sigma\in C^3_b(\R)$ then the function $\sigma$ satisfies Hypotheses \ref{hyp:non_lineare_2} with $\omega=2$.}
\end{example}

\subsection{The case when $\sigma'$ is locally Lipschitz continuous from $X_\alpha$ to $X$}
\label{subsect-3.3}

We conclude this section by proving that we can further weaken the assumptions on $\sigma$ allowing $\sigma'$ to be locally Lipschitz continuous from the smaller space $X_\alpha$ to $X$.
\begin{hypotheses}
\label{hyp:non_lineare_3}
\begin{enumerate}[\rm (i)]
\item
Hypotheses $\ref{hyp-main}$
are satisfied, with $0\le\zeta\le\lambda<2$.
\item
The function $x$ belongs to $C^{\eta}([0,1])$ for some $\eta\in (1/2,1)$.
\item 
The function $\sigma:X\to X$ is G\^ateaux differentiable with bounded G\^ateaux derivative $\sigma'$. 
\item
There exists $\alpha\in(0,1)$ such that $\eta+\alpha>1$, the restriction of $\sigma$ to $X_\alpha$ maps the space into itself, and there exist positive constants ${\rm Lip}_\sigma^\alpha$, ${\rm Lip}_{\sigma'}^\alpha$ and $\omega\geq 1$ such that 
\begin{align*}
& \|\sigma(x)-\sigma(y)\|_{X_{\alpha}}\leq {\rm Lip}_\sigma^\alpha (1+R)^{\omega}\|x-y\|_{X_\alpha},
\notag \\
& \|(\sigma'(x))(h)- (\sigma'(y))(h)\|_{X}\leq {\rm Lip}_{\sigma'}^\alpha(1+R)^{\omega-1}\|x-y\|_{X_\alpha}\|h\|_X
\end{align*}
for every $x,y\in X_\alpha$, with $\|x\|_{X_\alpha}$ and $\|y\|_{X_\alpha}\leq R$, every $R>0$ and $h\in X$.
Moreover,
$\|\sigma(x)\|_{X_\alpha}\leq L_\sigma^\alpha(1+\|x\|_{X_\alpha})$ for every $x\in X_{\alpha}$ and some positive constant $L_\sigma^\alpha$.
\end{enumerate}
\end{hypotheses}

\begin{theorem}
\label{thm:ex_mild_sol_3}
Under Hypotheses $\ref{hyp:non_lineare_3}$ the statements of Theorem $\ref{thm:ex_mild_sol_2}$ and Proposition $\ref{prop:reg_mild_sol_2}$ hold true with the same choice of all the parameters.
\end{theorem}

\begin{proof}
The only differences between this proof and the one of the quoted Theorem and Proposition is in the existence and uniqueness part of the statement. More precisely, we need to slightly modify the arguments used to prove the crucial estimate \eqref{macau-1} that is:
$$
\|\mathscr{I}_{S,\sigma\circ y_1-\sigma\circ y_2}(s,t)\|_{X_{\alpha}}\le \widetilde C|t-s|^{\eta-(1+\omega)(\alpha-\theta)}\|x\|_{C^{\eta}([0,T])}\|y_2-y_1\|_{Y^{\alpha}_{\theta-\alpha}(0,T_*)},
$$
which holds true for every $(s.t)\in [0,T_*]^2_{<}$ and some positive constant $\widetilde C$, 
depending on $\alpha$, $\theta$, $x$, $\eta$, $\mathfrak R$, $\sigma$ and $\omega$.
The starting point is still estimate \eqref{macau}.
The different assumptions on $\sigma'$ force us to estimate the term $\hat\delta_1(\sigma\circ y_1-\sigma\circ y_2)$ differently from what we did in  \eqref{stima_diff_contr_4}, to obtain that
\begin{align}
\label{stima_diff_lambda_nuova_3}
\|\hat \delta_1(\sigma\circ y_1-\sigma\circ y_2)\|_{\mathscr C^{\alpha}_{-(1+\omega)(\alpha-\theta)}((0,T_*]^2_<;X)}
\leq & \widetilde C_1\|y_1-y_2\|_{Y^{\alpha}_{\theta-\alpha}(0,T_*)}
\end{align}
for some positive constant $\widetilde C_1$ which depends on $\alpha,\theta,x,\eta$, $\mathfrak R$, $\sigma$ and $\omega$.

Let us fix $(s,t)\in(0,T_*]^2_<$. We recall that
\begin{align*}
& (\delta_1(\sigma\circ y_1-\sigma\circ y_2))(s,t) \\
= & \int_0^1( \sigma'(y_1(s)+r(\delta_1y_1)(s,t))-\sigma'(y_2(s)+r(\delta_1y_1)(s,t)))(\delta_1y_1)(s,t)dr \\
& +\sigma(y_2(s)+(\delta_1y_1)(s,t))-\sigma(y_2(s)+(\delta_1y_2)(s,t))=:\phi_1(s,t)+\phi_2(s,t),
\end{align*}
(see \eqref{isolaman}).
The term $\phi_2(s,t)$ is estimated by \eqref{stima_diff_contr_4}.
As far as $\phi_1$ is concerned, we get
\begin{align*}
&s^{(1+\omega)(\alpha-\theta)}\|\phi_1(s,t)\|_X\notag\\
\le &{\rm Lip}_{\sigma'}^{\alpha}(1+3\mathfrak{R}s^{\theta-\alpha})^{\omega-1}s^{(1+\omega)(\alpha-\theta)}\|y_1(s)-y_2(s)\|_{X_{\alpha}}\|(\delta_1y_1)(s,t)\|_X\notag\\
=&{\rm Lip}_{\sigma'}^{\alpha}(1+3\mathfrak R)^{\omega-1}s^{\alpha-\theta}\|y_1(s)-y_2(s)\|_{X_\alpha}s^{\alpha-\theta}\|(\delta_1y_1)(s,t)\|_X \notag \\
\leq & {\rm Lip}_{\sigma'}^{\alpha}(1+3\mathfrak R)^{\omega-1}\|y_1-y_2\|_{Y^{\alpha}_{\theta-\alpha}(0,T_*)}s^{\alpha-\theta}\|(\delta_1y_1)(s,t)\|_X
\end{align*}
and, arguing as in \eqref{stima_diff_contr_4}, we infer that
\begin{align*}
s^{\alpha-\theta}\|(\delta_1y_1)(s,t)\|_X
\leq & (K_{\alpha,0}+C_{\alpha,0})\|y_1\|_{Y^{\alpha}_{\theta-\alpha}(0,T_*)}|t-s|^\alpha. 
\end{align*}
Summing up, we have proved that
\begin{align*}
\|\delta_1(\sigma\circ y_1-\sigma\circ y_2)\|_{\mathscr C^{\alpha}_{-(1+\omega)(\alpha-\theta)}((0,T_*]^2_<;X)} 
\leq &  \widetilde C_2 \|y_1-y_2\|_{Y^{\alpha}_{\theta-\alpha}(0,T_*)}
\end{align*}
for some positive constant $\widetilde C_2$, which depends on $\|\psi\|_{X_\theta},\alpha,\theta,x,\eta$, $\mathfrak R$, $\sigma$ and $\omega$. 
From this estimate and \eqref{stima_diff_lambda_2}, \eqref{stima_diff_lambda_nuova_3} follows easily.

Now, we can complete the proof following the arguments in Theorem \ref{thm:ex_mild_sol_2}.
\end{proof}

\begin{remark}
{\rm Remark \ref{rmk-finale-mond} can be applied also to the results of this subsection.}
\end{remark}

\begin{example}
\label{example-3.18}
Let $A$ be the realization of the second-order derivative in $X=L^2((0,1))$, with homogeneous Dirichlet boundary conditions, and let $\sigma(\psi):=\hat\sigma\circ\psi$ for every $\psi\in L^2((0,1))$, where $\hat\sigma$ is a fixed function in $C_b^2(\R)$. In this situation, $X_{\lambda}=W^{2\lambda,2}((0,1))$ if $\lambda\le 1/2$, $X_{\lambda}=W^{2\lambda,2}((0,1))\cap W^{1,2}_0((0,1))$ if $\lambda\in (1/2,1)$, $X_1=D(A)=W^{2,2}((0,1))\cap W^{1,2}_0((0,1))$ and $X_{\lambda}=\{u\in W^{2,2}((0,1))\cap W^{1,2}_0((0,1)): u''\in X_{\lambda-1}\}$ if $\lambda\in (1,2)$.

Let $\alpha\in(\frac{1}{4},\frac{1}{2})$, $\theta\in[0,\alpha]$ and $\eta\in(\frac{1}{2},1)$ be such that $\eta>3\alpha-2\theta$. Under these conditions, Hypotheses $\ref{hyp:non_lineare_3}$ are satisfied with $\omega=1$. Indeed, it is easy to check that $\sigma$ is Lipschitz continuous from $X$ in itself. Moreover,
\begin{align*}
\|(\sigma'(y_2))(h)-(\sigma'(y_1))(h)\|^2_X
= & \int_0^1|\hat\sigma'(y_2(\xi))-\hat\sigma'(y_1(\xi))|^2|h(\xi)|^2d\xi \\
\leq & \|\hat\sigma\|^2_{C^2_b(\R)}\int_0^1|y_2(\xi)-y_1(\xi)|^2|h(\xi)|^2d\xi \\
\leq & \|\hat\sigma\|^2_{C^2_b(\R)}\|y_2-y_1\|^2_{C_b([0,1])}\|h\|_X^2 \\
\leq & \|\hat\sigma\|^2_{C^2_b(\R)}\|y_2-y_1\|^2_{X_{\alpha}}\|h\|_X^2, \end{align*}
for every $y_1,y_2\in X_\alpha$ and every $h\in X$,
since $X_\alpha\subset C_b([0,1])$ with a continuous embedding. Therefore, the second estimate in Hypothesis $\ref{hyp:non_lineare_3}$ is satisfied with $\omega=1$. To prove that $\sigma$ is locally Lipschitz continuous in $X_{\alpha}$  we observe that
\begin{align*}
(\sigma(y_1))(x)-(\sigma(y_2))(x)=(y_1(x)-y_2(x))\int_0^1\hat\sigma'(ry_1(x)+(1-r)y_2(x))dr    
\end{align*}
for every $y_1,y_2\in X_{\alpha}$ and almost every $x\in (0,1)$.
From this formula it follows that
\begin{align*}
\|\sigma(y_1)-\sigma(y_2)\|_{L^2((0,1))}\le \|\hat\sigma'\|_{\infty}\|y_2-y_1\|_{L^2((0,1))}.
\end{align*}
Moreover,
\begin{align*}
&(\sigma(y_1))(\xi)-(\sigma(y_2))(\xi)-
(\sigma(y_1))(\eta)+(\sigma(y_2))(\eta)\\
=& [y_1(\xi)-y_2(\xi)-y_1(\eta)+y_2(\eta)]\int_0^1\hat\sigma'(ry_1(\xi)+(1-r)y_2(\xi))dr\\
&+(y_1(\eta)-y_2(\eta))\int_0^1[\hat\sigma'(ry_1(\xi)+(1-r)y_2(\xi))-\hat\sigma'(ry_1(\eta)+(1-r)y_2(\eta))]dr
\end{align*}
so that
\begin{align*}
&|(\sigma(y_1))(\xi)-(\sigma(y_2))(\xi)-
(\sigma(y_1))(\eta)+(\sigma(y_2))(\eta)|^2\\
\le & 2\|\hat\sigma'\|_{\infty}^2|y_1(\xi)-y_2(\xi)-y_1(\eta)+y_2(\eta)|^2\\
&+2\|\hat\sigma''\|_{\infty}^2|y_1(\eta)-y_2(\eta)|^2(|y_1(\xi)-y_1(\eta)|^2+|y_2(\xi)-y_2(\eta)|^2)
\end{align*}
for almost every $\xi,\eta\in (0,1)$. Consequently,
we can estimate
\begin{align*}
&[\sigma(y_1)-\sigma(y_2)]_{W^{2\alpha,2}(0,1)}^2\\
\le &2\|\hat\sigma'\|_{\infty}^2[y_1-y_2]_{W^{2\alpha,2}((0,1))}^2
+2\|\hat\sigma''\|_{\infty}^2\|y_1-y_2\|_{\infty}^2
([y_1]_{W^{2\alpha,2}((0,1))}^2+[y_2]_{W^{2\alpha,2}((0,1))}^2)\\
\le &C_*(1+\|y_1\|_{X_{\alpha}}^2+\|y_2\|_{X_{\alpha}}^2)\|y_1-y_2\|_{X_{\alpha}}^2
\end{align*}
for some positive constant $C_*$, independent of $y_1$ and $y_2$.
We have so proved that $\sigma$ is locally Lipschitz continuous on $X_\alpha$ and the first condition in Hypotheses $\ref{hyp:non_lineare_3}$ is satisfied with $\omega=1$.

Note that, if $3\alpha<1$ and $\eta>3\alpha$ then we can take $\theta=0$, i.e., problem \eqref{omo_equation} admits a unique mild solution with initial datum $\psi\in X=L^2((0,1))$.

Finally, we observe that, if $\hat\sigma\in C^3_b(\R)$ and $\alpha\in (1/2,1)$, then the function $\sigma$ satisfies Hypotheses $\ref{hyp:non_lineare_3}$ with $\omega=2$.
\end{example}

\end{document}